\documentclass[12pt,a4paper]{article}
\usepackage[utf8]{inputenc}
\usepackage[OT1]{fontenc}
\usepackage{amsmath}
\usepackage{amsthm}
\usepackage{amsfonts}
\usepackage{amssymb}
\usepackage{graphicx}
\usepackage{cite}
\usepackage{hyperref}
\usepackage{xcolor}
\usepackage[left=2cm,right=2cm,top=2cm,bottom=2cm]{geometry}
\graphicspath{{pic/}}

\newtheorem{lemma}{Lemma}[section]
\newtheorem{proposition}[lemma]{Proposition}
\newtheorem{theorem}[lemma]{Theorem}

\newtheorem{definition}[lemma]{Definition}

\theoremstyle{definition}
\newtheorem{remark0}[lemma]{Remark}

\newtheorem*{solution*}{Solution}

\newenvironment{remark}{\begin{remark0}\rm}{\hspace*{\fill} $\square$
                        \end{remark0}}

\DeclareMathOperator{\End}{End}
\DeclareMathOperator{\gr}{gr}

 \DeclareMathOperator{\ad}{ad}
  \DeclareMathOperator{\Ad}{Ad}

 \newcommand{\bbar}[1]{\setbox0=\hbox{$#1$}\dimen0=.2\ht0 \kern\dimen0 \overline{\kern-\dimen0 #1}}

 \let\hom\relax % kills the old hom, which is lowercase
 \DeclareMathOperator{\hom}{Hom}
 \DeclareMathOperator{\Hom}{\mathcal{H}\kern-.125em\mathpzc{om}}
 
 \DeclareMathOperator{\id}{id}
 \DeclareMathOperator{\im}{im}

 \DeclareMathOperator{\Proj}{\mathcal{P}\kern-.125em\mathpzc{roj}}

  \DeclareMathOperator{\tr}{tr}
 \newcommand{\udot}{\ensuremath{{\lower .183333em \hbox{\LARGE \kern -.05em$\cdot$}}}}

  \DeclareMathOperator{\divv}	{div} 
   \DeclareMathOperator{\grad}	{grad}

 \newcommand{\cD}{\mathcal{D}}

 \newcommand{\cL}{\mathcal{L}}
 
 \newcommand{\cN}{\mathcal{N}}
 \newcommand{\cO}{\mathcal{O}}
\newcommand{\cP}{\mathcal{P}}

\newcommand{\cS}{\mathcal{S}}

 \newcommand{\N}{\mathbb{N}}

 \newcommand{\R}{\mathbb{R}}

 \newcommand{\T}{\mathbb{T}}

 \newcommand{\fg}{\mathfrak{g}}

 \newcommand{\fh}{\mathfrak{h}}

 \newcommand{\fn}{\mathfrak{n}}

 \newcommand{\p}{\partial}

 \DeclareMathOperator{\spann}{span}
 
\newcommand{\of}{\mathcal{O}}

\newcommand{\dd}{\,{\mathrm d}}
\newcommand{\db}{{\mathrm d}}
\newcommand{\st}{\,\p}

\makeatletter
\let\@fnsymbol\@arabic
\makeatother

\numberwithin{equation}{section}
\title{Cartan connections for stochastic developments on sub-Riemannian
  manifolds}
\author{Ivan Beschastnyi\footnote{{\sc Laboratoire Jacques-Louis Lions,
    Sorbonne Universit{\'e}, Universit{\'e} de Paris, CNRS, Inria,
    75005 Paris, France.} {\it E-mail address:}
  \texttt{i.beschastnyi@gmail.com}}\and
Karen Habermann\footnote{{\sc Department of Statistics, University of
    Warwick, Coventry, CV4 7AL, United Kingdom.} {\it E-mail address:}
  \texttt{karen.habermann@warwick.ac.uk}}\and
Alexandr Medvedev\footnote{{\sc International School for Advanced Studies,
  via Bonomea 265, Trieste 34136, Italy.} {\it E-mail address:}
  \texttt{sasha.medvedev@gmail.com}}}
 
\begin{document}

\maketitle

\begin{abstract}
Analogous to the characterisation of Brownian motion on a
Riemannian manifold as the development of Brownian motion on a
Euclidean space, we construct sub-Riemannian diffusions
on equinilpotentisable sub-Riemannian manifolds
by developing a canonical stochastic process arising as the lift
of Brownian motion to an associated model space. The notion of
stochastic development we introduce for equinilpotentisable
sub-Riemannian manifolds
uses Cartan connections, which take the place of the Levi--Civita
connection in Riemannian geometry.
We first derive a general expression
for the generator of the stochastic process which is the
stochastic development with respect to a Cartan connection of the
lift of Brownian motion to the model space.
We further provide a necessary and sufficient condition for the
existence of a Cartan connection which develops the canonical
stochastic process to the sub-Riemannian diffusion associated with
the sub-Laplacian defined with respect to the Popp volume. We
illustrate the construction of a suitable Cartan connection for free
sub-Riemannian structures with two generators and we discuss an
example where the condition is not satisfied.
\end{abstract}

\section{Introduction}
\label{sec:intro}
Brownian motion, also called Wiener process, is a mathematical description
of the
animated and irregular motion of particles which
are suspended, say, in a fluid. This process plays an important role
in various areas of mathematics and is used, among others, to describe more
complicated stochastic processes, to model unknown forces in control
theory, to give a rigorous path integral formulation of quantum
mechanics, and it prominently features in mathematical finance.

Brownian motion $(b_t)_{t\geq 0}$ on a smooth Riemannian manifold $M$
with the Laplace--Beltrami operator $\Delta_M$
is the unique continuous-time stochastic process on $M$ whose
infinitesimal motion is described by $\frac{1}{2}\Delta_M$, that is,
for the heat semigroup $(P_t)_{t\geq 0}$ associated with
$(b_t)_{t\geq  0}$ and for any function $f\in C^\infty(M)$, we have
\begin{equation*}
  \frac{1}{2}\Delta_M f(q)=\lim_{t\downarrow 0}\frac{P_tf(q)-f(q)}{t}\;.
\end{equation*}
We call $\frac{1}{2}\Delta_M$ the infinitesimal generator of Brownian
motion on $M$.
One of the many interesting features of Brownian motion is that it can
be used to give a solution to the Dirichlet problem associated with
$\Delta_M$ on a domain in $M$. In fact, it is even possible to
uniquely characterise Brownian motion via the heat equation.
Brownian motion also arises as the limit of a sequence of random walks
on the manifold $M$.

Yet another alternative construction of Brownian motion uses the
notion of anti-development of a curve in a Riemannian manifold.
To a differentiable curve in the Riemannian manifold $M$ of dimension
$n$, we can associate a curve in the model space $\R^n$
via the Levi--Civita connection. By extending this correspondence to
stochastic processes whose
sample paths are almost surely continuous but nowhere
differentiable it can be shown that a stochastic process on the manifold
$M$ of dimension $n$ is a Brownian motion on $M$ if and only if its
anti-development is a standard Brownian motion on $\R^n$. In
particular, if we take a standard Brownian motion on $\R^n$ we obtain
a process in the orthonormal frame bundle $\of(M)$ which projects
nicely to $M$ to give a Brownian motion on the Riemannian manifold $M$.

For further details on and properties of Brownian motions on smooth
Riemannian manifolds, see, for instance,
\'{E}mery~\cite{emery}, Grigor'yan~\cite{grig}, Hsu~\cite{hsu},
and J{\o}rgensen~\cite{jorgensen}, as well as~\cite{part3essay} for a
more exhaustive overview of the various characterisations of Brownian
motion.

With Brownian motions on smooth Riemannian manifolds being well
understood, we turn our attention to the sub-Riemannian setting. A
sub-Riemannian manifold is a triple $(M,\cD,g)$ consisting of a smooth
manifold $M$ together with a bracket generating distribution
$\cD\subset TM$ and a metric $g$ on $\cD$. As the natural
generalisation of the Laplace--Beltrami operator in Riemannian
geometry, the sub-Riemannian Laplacians, also called sub-Laplacians,
on a sub-Riemannian manifold
are defined as the divergence of the horizontal gradient. The
divergence $\divv_\nu$ depends on a choice of a positive smooth
measure $\nu$ on the manifold $M$, and the horizontal gradient
$\grad_H f$ of a function $f\in C^\infty(M)$ is a smooth section of
$\cD$ such that, for any section $X\in \Gamma(\cD)$,
\begin{equation*}
  g(\grad_H f,X)=\db f(X)\;.
\end{equation*}
The sub-Laplacian $\Delta_\nu$ with respect to the measure
$\nu$ acting on smooth functions $f$ on $M$ is thus given by
\begin{equation*}
  \Delta_\nu f = \divv_\nu(\grad_H f)\;.
\end{equation*}
For a local frame $(X_1,\dots,X_{k_1})$ of $\cD$ near points of maximal
  rank which is orthonormal with
respect to $g$, we can locally write 
\begin{equation*}
  \Delta_\nu = \sum_{i=1}^{k_1} X_i^2 + \sum_{i=1}^{k_1} \divv_\nu(X_i)X_i\;.
\end{equation*}

Similar to Brownian motion on a Riemannian manifold, we can consider
the continuous-time stochastic process on $M$ whose infinitesimal
generator is $\frac{1}{2}\Delta_\nu$. These processes, which we call
sub-Riemannian diffusions, are significantly less well
understood than Brownian motion, and many of its properties
remain to be understood.

Following M\'{e}tivier~\cite{metivier} and Ben
Arous~\cite{GBA1,GBA2}, the small-time asymptotics of
sub-Riemannian diffusion processes have been studied, amongst others, by
Bailleul, Mesnager, Norris~\cite{BMN}, by Barilari, Boscain,
Neel~\cite{BBN}, by de Verdière, Hillairet, Trélat~\cite{trelat} as well
as in~\cite{PhDthesis} and \cite{MAP}. Already at the level of
small-time asymptotics, it is seen that sub-Riemannian diffusions show
qualitatively different behaviours compared to Brownian motions.
We further remark that since sub-Laplacians are in
divergence form, the first-order small-time heat kernel asymptotics of
the associated
stochastic processes only depend on the underlying sub-Riemannian
structure, and not on the measure $\nu$.

For certain sub-Laplacians, the question of approximating
the associated sub-Riemannian diffusions by random
walks has been addressed by Boscain, Neel, Rizzi~\cite{BNR} and by
Gordina, Laetsch~\cite{gordina}.

In comparison to the various characterisations of Brownian motions, the
one characterisation which appears to be missing for sub-Riemannian
diffusions is as the development of a suitable model stochastic
process. The first major obstacle to such a construction is that the
notion of Levi--Civita connection does not carry over to the
sub-Riemannian setting. Instead, we employ the notion of Cartan
connections which is well-adapted to the graded structures 
appearing in the study of sub-Riemannian manifolds. The Cartan
geometry approach to sub-Riemannian geometry works particularly well
for equinilpotentisable sub-Riemannian manifolds. Moreover, for
equiregular, and therefore also for equinilpotentisable,
sub-Riemannian manifolds, there exists a smooth volume canonically
associated with the sub-Riemannian structure, which is the so-called
Popp volume. The objective of this article is to initiate the
characterisation of sub-Riemannian diffusions via stochastic
development by providing such a construction on a wide range of
equinilpotentisable
sub-Riemannian manifolds for the
sub-Riemannian diffusion associated with the sub-Laplacian
defined with respect to the Popp volume.

We stress that the work presented lies more on the differential geometry side
and is mainly concerned with constructing suitable Cartan
connections.
The central ingredient needed from stochastic analysis is
the correspondence between Stratonovich stochastic differential equations
and infinitesimal generators, which is provided at the suitable point.

Cartan geometry makes the idea of a model tangent space rigorous. For
example, the tangent space of a Riemannian manifold is a Euclidean
space and we can view a Cartan connection as a way of rolling this
Euclidean space on the Riemannian manifold,
cf. Wise~\cite{roll_cartan}. It gives rise to a natural notion of
the development of a curve $\gamma\colon [0,t]\to \R^n$ to $M$ by
simply saying that the contact point while rolling traces out
$\gamma$, and the notion of stochastic development is defined in a
similar way.

In the sub-Riemannian setting, the model tangent spaces are Lie
algebras of nilpotent Lie groups known as the Carnot groups. The
distribution $\cD$ of a
sub-Riemannian manifold $(M,\cD,g)$ generates a filtration which is
defined iteratively at $q\in M$, for $i\in\N$, by $\cD^{-1}_q = \cD_q$
and
\begin{equation*}
  \cD^{-(i+1)}_q = \cD^{-i}_q + \left[\cD,\cD^{-i}\right]_q\;.
\end{equation*}
The minimal $m\in\N$ such that $\cD^{-m}_q = T_q M$ for every $q\in M$
is the step of the sub-Riemannian manifold, and the tuple of numbers
$(k_1,k_2,\dots,k_m) = (\dim \cD_q^{-1},\dim
\cD_q^{-2},\dots,\dim\cD_q^{-m})$ is called the growth vector at $q\in
M$. Using the filtration, we further define the associated grading of
the tangent space $T_q M$ at $q\in M$ by
\begin{equation}\label{grref}
  \gr\left(T_qM\right) = 
  \cD^{-1}_q \oplus \cD^{-2}_q/\cD^{-1}_q \oplus \dots
  \oplus \cD^{-m}_q/\cD^{-(m-1)}_q\;.
\end{equation}
Each $\gr(T_qM)$ has the natural structure of a nilpotent Lie
algebra as well as a natural horizontal metric defined on
$\cD^{-1}_q$. If $X \in \Gamma(\cD^{-i})$, $Y \in \Gamma(\cD^{-j})$,
for $i,j\in\{1,\dots,m\}$, are sections of the corresponding bundles,
then the Lie algebra structure is defined by
\begin{equation*}
  \left[X+\Gamma\left(\cD^{-i+1}\right),
    Y+\Gamma\left(\cD^{-j+1}\right)\right]
  = [X,Y]+\Gamma\left(\cD^{-i-j+1}\right)\;.
\end{equation*}

We notice that the metric $g$ on $\cD$ induces a metric on all of
$\gr(T_qM)$ for $q\in M$ as follows. On the tensor product
$\otimes_{i=1}^l \cD^{-1}_q$ for $l\in\{2,\dots,m\}$, we define a map
\begin{equation*}
  \pi_l\colon
  \bigotimes_{i=1}^l \cD^{-1}_q \to \cD^{-l}_q/\cD^{-l+1}_q
\end{equation*}
given, for vector fields $X_1,\dots,X_l$ on $M$ extending
vectors $v_1,\dots,v_l \in T_q M$, by
\begin{equation*}
  \pi_l(v_1 \otimes \dots \otimes v_l) =
  [X_1,[X_2,\dots,[X_{l-1},X_l]\dots]](q){} \mod \cD_q^{-l+1}\;.
\end{equation*}
Using the metric induced by $g$ on $\otimes_{i=1}^l \cD^{-1}$, we
identify $\cD^{-l}_q/\cD^{-l+1}_q$ with $(\ker \pi_l)^\perp$ for the
restricted metric. This further
gives rise to a metric on the whole of $\gr(T_q M)$, which plays an
important role in Section~\ref{sec:cartan}. The
space $\Lambda^n(T_q M)$ is then naturally isomorphic to
$\Lambda^n\gr(T_q M)$, and the constructed metric allows us to choose
a canonical element in $\Lambda^n\gr(T_q M)$. The image of this
element under the canonical isomorphism is the Popp volume at
$q\in M$. For further details, see
Montgomery~\cite[Chapter~10]{montgomery}.

A sub-Riemannian manifold is called equinilpotentisable if
$\gr(T_q M)$ does not depend on the point $q\in M$ as a metric Lie
algebra, that is, all the
associated gradings are metrically isomorphic. In this case, we write
$\fn_{-k} = \cD^{-k}_q/\cD^{-k+1}_q$ for $k\in\{1,\dots,m\}$ as well as
$\fn = \gr(T_q M)$ to
emphasise this independence. The Carnot group which serves as
model tangent space to an equinilpotentisable sub-Riemannian manifold
is the simply connected Lie group whose Lie algebra is $\fn$.
In the Cartan terminology, this Carnot group is called nilpotent
model. It represents the flat space and the curvature invariants measure
how much a given sub-Riemannian manifold differs from its nilpotent
model.

It is important to note that Carnot groups can possess several
non-equivalent left-invariant sub-Riemannian metrics. We say that
left-invariant sub-Riemannian metrics $g_1$ and $g_2$ on a Carnot
group $(G, \cD)$ are equivalent if there exists a graded Lie group
automorphism that maps $g_1$ to $g_1$. 
One of the simplest examples of a Carnot group with non-equivalent
sub-Riemannian metrics is the 5-dimensional Heisenberg group $H_5$.
Its Lie algebra $\fh_5$ has a natural filtration $\R^4 \oplus z \R =
\cD^{-1} \oplus \cD^{-2}$, where $z$ is an element of the centre of
$\fh_5$ and where, for $v_1,v_2\in \cD$, the Lie brackets are given,
in terms of a symplectic form $\omega$ on $\cD$, by
\begin{equation*}
  [v_1, v_2] = \omega(v_1, v_2) z\;.
\end{equation*}
All automorphisms that preserve $\cD=\mathbb R^4$ have to preserve
$\omega$ up to multiplication by a scalar. Therefore, the group of
$\cD$-preserving automorphisms is equivalent to the
conformal-symplectic group $CSp(\mathbb R^4)$. As all symplectic forms on
$\mathbb R^4$ are equivalent, classification of metrics
on $\mathbb R^4$ with respect to $CSp(\mathbb R^4)$ is equivalent to
the classification of pairs $(g, [\omega])$ where $g$ is a metric on
$\R^4$ and $[\omega]$ is a conformal class of symplectic forms. By
normalising $g$ to an arbitrary fixed metric, our classification
problem reduces to the classification of $[\omega]$ with respect to
the orthogonal group $O(\mathbb R^4)$. Furthermore, instead of
$\omega$ we consider the skew-symmetric operator $\omega \circ
g^{-1}$. Using the action of $O(\mathbb R^4)$, we can normalise every
skew-symmetric operator to the form
\begin{equation*}
\begin{pmatrix} 
0 & \lambda_1 & 0 & 0\\
-\lambda_1 & 0 & 0 & 0 \\
0 & 0 & 0 & \lambda_2 \\
0 & 0 & -\lambda_2 & 0 \\
\end{pmatrix},
\end{equation*}
where $\lambda_1 \ge \lambda_2 > 0$.
We see that conformal classes $[\omega]$ are in one-to-one
correspondence with the ratio $\lambda_1:\lambda_2$. This means that
the family of non-equivalent metrics on $H_5$ is one-dimensional.

A Cartan connection allows us to develop curves from the
corresponding Carnot 
group to an equinilpotentisable
sub-Riemannian manifold, and in the same spirit to define a notion of
stochastic development, exactly like it was done in the
Riemannian case.

There is a canonical sub-Riemannian diffusion on a Carnot group,
arising as the lift of a Brownian motion on $\R^{k_1}$ and
taking the place of the standard Brownian motion on $\R^n$, which we
develop to give a sub-Riemannian diffusion on the corresponding
sub-Riemannian manifold. The generator of the resulting stochastic
process always has the same second order term, while the first order
term depends on the choice of the Cartan connection. As said
previously, we are particularly
interested in constructing Cartan connections which give rise to
sub-Riemannian diffusions associated to sub-Laplacians
defined with respect to the Popp volume.

Barilari and Rizzi~\cite{dave_luca} give a local formula for the Popp
volume $\cP$ and for the sub-Laplacian $\Delta_\cP$ with respect to the
Popp volume in terms of an adapted frame, that is, a frame
$(X_1,\dots,X_n)$ such that $X_1,\dots,X_{k_i}$ span $\cD^{-i}$ for
$i\in\{1,\dots,m\}$. On an equinilpotentisable sub-Riemannian
manifold, the local formula for $\Delta_\cP$ only involves the structure
constants $c_{ij}^k\in C^\infty(M)$ of the adapted frame, which
satisfy, for $i,j\in\{1,\dots,n\}$,
\begin{equation*}
  [X_i,X_j]=\sum_{k=1}^nc_{ij}^k X_k\;,
\end{equation*}
and it is given by
\begin{equation}
  \label{eq:sr_laplacian}
  \Delta_\cP =
  \sum_{i=1}^{k_1} \left( X_i^2 - \sum_{l=1}^n c^l_{il} X_i \right)\;.
\end{equation}

Restricting our attention to equinilpotentisable sub-Riemannian
manifolds limits the classes of structures that we can consider. For
instance, even contact structures in higher dimensions are not
equinilpotentisable in general, and our setting excludes singular
structures like the Martinet manifold. Moreover, even for a
equinilpotentisable sub-Riemannian manifold, there does not always
exist a Cartan connection which allows us to characterise the
sub-Riemannian diffusion associated with $\Delta_\cP$
in terms of a stochastic development. However, in
Theorem~\ref{thm:develop}, we provide a necessary and sufficient
condition, which is proven in Section~\ref{sec:cartan}, for when we
can obtain such Cartan connections.
All the Cartan connections we construct are
characterised by a torsion-free-like condition, which requires
a certain part of the curvature two-form to vanish.

We remark that Baudoin, Feng, Gordina~\cite{bott}
consider the stochastic parallel transport with respect to the Bott
connection on foliated manifolds, and that Angst, Bailleul,
Tardif~\cite{kinetic} use the Cartan geometry approach to define
kinetic Brownian motion on Riemannian manifolds.

A Cartan connection in the sub-Riemannian setting is a
$\fg$-valued one-form $\omega$, where $\fg$ is a semi-direct product
of the nilpotent Lie algebra $\fn$ and the Lie algebra $\fh$ of the
Lie group $H$ of infinitesimal symmetries of $\fn$ which is isomorphic to
a subgroup of
$SO(k_1)$. The connection form can be separated into its
$\fn$-valued part $\omega_\fn$ and its $\fh$-valued part $\omega_\fh$.
We choose a basis $\{A_\alpha\colon 1\leq\alpha\leq \dim H\}$ of $\fh$
and consider the corresponding components $\omega^\alpha$ of
$\omega_\fh$. If we further choose a local coframe
$(\theta^1,\dots,\theta^n)$ of $T^*M$ dual to the adapted frame
$(X_1,\dots,X_n)$, we can locally write
\begin{equation*}
  \omega^\alpha = \sum_{i=1}^n\Gamma^\alpha_i \theta^i \;,
\end{equation*}
for some smooth functions $\Gamma^\alpha_i$ which are the Christoffel
symbols of the connection $\omega_\fh$. Often we aggregate the
first $k_1$ components into a single vector-valued function
\begin{equation*}
  \Gamma^\alpha = \begin{pmatrix}
    \Gamma^\alpha_1\\
    \vdots\\
    \Gamma^\alpha_{k_1}
  \end{pmatrix}\;,
\end{equation*}
and similarly, we write
\begin{equation*}
  X = \begin{pmatrix}
    X_1\\
    \vdots\\
    X_{k_1}
  \end{pmatrix}  
\end{equation*}
to shorten the notations and to simplify the formulae.

The main formula underlying our constructions of Cartan
connections is provided in the following theorem.
\begin{theorem}\label{thm:main_formula}
For an equinilpotentisable sub-Riemannian manifold $(M,\cD,g)$
with a symmetry group $H$
and an adapted Cartan connection $\omega$,
the generator $\frac{1}{2}\Delta$ of the
stochastic process
on $M$ obtained as the stochastic development of the
canonical
sub-Riemannian diffusion on the nilpotent model $N$ is given,
in a local trivialisation and in
terms of the Christoffel symbols $\Gamma^\alpha$, by
\begin{equation}
  \label{eq:main_formula}
  \Delta = \sum_{i=1}^{k_1} X_i^2 + \sum_{\alpha=1}^{\dim H}
  (\Gamma^\alpha)^T A_\alpha X\;.
\end{equation}
\end{theorem}

While the formula in the theorem above is expressed in a local
trivialisation, the operator $\Delta$ and the stochastic development
of the canonical sub-Riemannian diffusion on the nilpotent group $N$
are coordinate invariant objects. 

The key result of this article is stated in the next theorem.

\begin{theorem}\label{thm:develop}
Suppose $(M,\cD,g)$ is an equinilpotentisable sub-Riemannian
manifold with a symmetry group $H$. Then there exists a Cartan
connection $\omega$ such that the stochastic process arising as the
stochastic development of the canonical
sub-Riemannian diffusion on the nilpotent model $N$ has generator
$\frac{1}{2}\Delta_\cP$ if and only if every
one-dimensional
sub-representation of $H$ on $\fn_{-1}$ corresponds to a divergence-free
vector field in $\cD$.
\end{theorem}
The proof of Theorem~\ref{thm:develop} is constructive and it gives an
explicit description for a possible choice of a Cartan
connection. The result particularly implies that there exist suitable Cartan
connections for many interesting examples including all
equinilpotentisable contact structures and structures with the growth
vector $(2,3,5)$. The latter arise when rolling distributions of surfaces, and
it can be applied to modelling rolling of spherical robots on an
unknown non-flat ground such as soil.

\medskip

The article has the following structure.
In Section~\ref{sec:cartan_geom}, we provide an overview of Cartan
geometry, with a focus on Cartan geometry on
a sub-Riemannian manifold in Section~\ref{sec:scartan}, and
as an illustration we show how to understand the Levi--Civita
connection on a Riemannian manifold as a Cartan connection. In
Section~\ref{sec:develop}, we discuss how the development with
respect to a Cartan connection of a curve in the model space can be
characterised by a system of ordinary differential equations, and
we use this to introduce a notion of stochastic development for
equinilpotentisable sub-Riemannian manifolds.
We further prove Theorem~\ref{thm:main_formula} and we show that
the Cartan geometry approach recovers the characterisation of Brownian
motion on a Riemannian manifold via stochastic development.
In Section~\ref{sec:cartan}, we establish
Theorem~\ref{thm:develop}, and we
illustrate the construction of a suitable Cartan connection
for manifolds modelled by free nilpotent
structures with two generators, which include 3D contact
structures. We further provide an example of a manifold where
a stochastic development of the canonical sub-Riemannian diffusion on
the model space never gives rise to the stochastic process with
generator $\frac{1}{2}\Delta_\cP$.

\proof[Acknowledgement]
The first author was supported by the ANR project Quaco
  ANR-17-CE40-0007-01.
The second author was supported by the Fondation Sciences
Math{\'e}matiques de Paris.

\section{Overview of Cartan geometry}
\label{sec:cartan_geom}
We provide a general overview of Cartan geometry with a focus
towards Cartan geometry on a sub-Riemannian manifold. In this section,
we use the Einstein summation convention.

Before giving the general definitions, we start by interpreting
Riemannian geometry as a Cartan geometry. Let us consider a
Riemannian manifold $(M,g)$ of dimension $n$ together with a
connection $\nabla$ on the tangent bundle $TM$. In a local
orthonormal frame $(X_1,\dots,X_n)$ of $TM$, the connection $\nabla$
is uniquely characterised by the Christoffel symbols
$\Gamma^k_{ij}$, for $i,j,k\in\{1,\dots,n\}$, which are given by, for
$i,j\in\{1,\dots,n\}$,
\begin{equation*}
  \nabla_{X_i} X_j = \Gamma^k_{ij}X_k\;.
\end{equation*}
If $(\theta^1,\dots,\theta^n)$ is a local coframe dual to
$(X_1,\dots,X_n)$, we can equivalently define a connection via
one-forms $\theta_i^j$, for $i,j\in\{1,\dots,n\}$, which satisfy
\begin{equation*}
  \nabla_{X} X_i = \theta^j_i(X) X_j\;.
\end{equation*}
Comparing this with the previous definition of a connection in terms
of Christoffel symbols, we find that, for all $i,j\in\{1,\dots,n\}$,
\begin{equation*}
  \theta^j_i = \Gamma^j_{ki}\theta^k\;.
\end{equation*}

The Levi--Civita connection is the unique torsion-free connection on $TM$
which is metric. For this connection, the metric
compatibility condition implies that, for all $i,j,k\in\{1,\dots,n\}$,
we have
\begin{equation*}
  \Gamma^k_{ij}+\Gamma^j_{ik} = 0\;,
\end{equation*}
which in turn gives the antisymmetry condition that, for all
$i,j\in\{1,\dots,n\}$,
\begin{equation*}
  \theta_i^j + \theta_j^i = 0\;.
\end{equation*}
To describe the torsion-free property of the Levi--Civita
connection in terms of the dual frame $(\theta^1,\dots,\theta^n)$, we
construct the Lie-algebra-valued one-form $\theta_{\mathfrak{h}}\in
\Omega(M,\mathfrak{so}(n))$ and the one-form $\theta_{\fn}\in
\Omega(M,\R^n)$ defined by, for $i,j\in\{1,\dots,n\}$,
\begin{equation*}
  \left(\theta_{\mathfrak{h}}\right)^i_j=\theta^i_j
  \quad\mbox{and}\quad
  \left(\theta_{\fn}\right)^i = \theta^i\;.
\end{equation*}
We call $\theta_{\mathfrak{h}}$ the Levi--Civita gauge and
$\theta_{\fn}$ the soldering gauge. Given a vector space $V$,
we define an exterior product between 
$\End(V)$-valued one-forms and $V$-valued one-forms by requiring that,
for $i,j\in\{1,\dots,n\}$ and all $A_i \in \End(V)$ as well as all
$v_j \in V$, we have
\begin{equation*}
  \left(A_i \otimes \theta^i\right)\wedge
  \left(v_j \otimes \theta^j\right) = A_i(v_j)
  \theta^i \wedge \theta^j\;.
\end{equation*}

The Levi--Civita gauge $\theta_{\mathfrak{h}}$ and the soldering gauge
$\theta_{\fn}$ characterise the torsion-freeness of a metric
connection as follows, see e.g. Sharpe~\cite{sharpich}.
\begin{proposition}
  A metric connection $\nabla$ is torsion-free if and only if for any
  orthonormal coframe $(\theta^1,\dots,\theta^n)$ the structure equations
  \begin{equation}
  \label{eq:torsion}
    \db\theta_\fn + \theta_{\fh}\wedge\theta_{\fn}=0
  \end{equation}
  are satisfied, that is, for all $i,j\in\{1,\dots,n\}$, we have
  \begin{equation*}
    \db\theta^i + \theta_j^i \wedge \theta^j = 0\;.
  \end{equation*}
\end{proposition}

Since the structure equations have to  hold for any frame, let us
consider what happens in a different frame. Applying a rotation
$h\in C^1(M,SO(n))$ to the frame with respect to which
$\theta_{\mathfrak{h}}$ and $\theta_{\fn}$ are defined yields
a frame in which $\theta_{\fn}^{\mathrm{new}}$ and
$\theta_{\fh}^{\mathrm{new}}$ are given by
\begin{align}
  \theta_{\fn}^{\mathrm{new}} &= h^{-1}\theta_{\fn}\;,\nonumber\\
  \theta^{\mathrm{new}}_{\fh} &= h^{-1}\dd h + h^{-1}\theta_{\fh} h\;.\label{eq:form_h}
\end{align}
These expressions show that $\theta_{\fh}$ is a pull-back of a
principle $SO(n)$-connection via a trivialising section of the
orthonormal frame bundle $\cO(M)$, that is, if $s\colon M \to
\cO(M)$ is the local section of $\cO(M)$ defined by the local
orthonormal frame $(X_1,\dots,X_n)$ then there exists a principle
$SO(n)$-connection with one-form $\omega_\fh$ such that
\begin{equation*}
  \theta_\fh = s^* \omega_\fh\;.
\end{equation*}
When changing the local section from $s$
to $sh$, we obtain $\theta^{\mathrm{new}}_{\fh}$ given
by~\eqref{eq:form_h}.
Similarly, the soldering gauge $\theta_{\fn}$ is a
  pull-back of the canonical soldering form $\omega_{\fn}$, that is,
\begin{equation*}
  \theta_\fn = s^* \omega_\fn\;.
\end{equation*}
The canonical soldering form $\omega_{\fn}$ can be defined on $\cO(M)$
in a totally
invariant manner. Let $\pi\colon \of(M)\to M$ be the projection
mapping. Then, for all $v\in T\of(M)$ and all
$f\in\of(M)$, we set
\begin{equation}
  \label{eq:sold_riem}
  \left(\omega_{\fn}\right)_f(v) =
  f^{-1}\dd\pi(v)\;,
\end{equation}
where $f \in \of(M)$ is considered as a map
$f\colon \R^n \to T_{\pi(f)} M$.

We combine $\omega_\fh$ and $\omega_{\fn}$ into a single matrix-valued
one-form $\omega$ on the frame bundle $\of(M)$ given by
\begin{equation*}
  \omega =
  \begin{pmatrix}
    \omega_{\fh} & \omega_{\fn}\\
    0 & 0
  \end{pmatrix}\;.
\end{equation*}
This one-form is an example of a Cartan connection. It takes values in the
Lie algebra $\mathfrak{se}(n)$ corresponding to the special Euclidean
Lie group $SE(n)$.
The curvature two-form $\Omega$ associated with a Lie-algebra-valued
one-form $\omega$
is given by
\begin{equation}
  \label{eq:curvature}
  \Omega = \db\omega + \frac{1}{2}[\omega,\omega]\;,
\end{equation}
for $[\cdot,\cdot]$ the commutator of two Lie-algebra-valued
one-forms, which is defined by
\begin{equation*}
  [\omega_1,\omega_2](X,Y) =
  [\omega_1(X),\omega_2(Y)]+[\omega_2(X),\omega_1(Y)]\;. 
\end{equation*}
In particular, we see that
\begin{equation}\label{eq:commute}
  [\omega_1,\omega_2]= [\omega_2,\omega_1]\;,
\end{equation}
and in a local trivialisation $(\theta^1,\dots,\theta^n)$, we have
\begin{equation}\label{eq:formula_com}
  \left[A_i \otimes \theta^i, B_j \otimes \theta^j\right]
  = \left[A_i,B_j\right]\otimes \theta^i \wedge \theta^j \;.
\end{equation}

The condition~\eqref{eq:torsion} is then equivalent to
the vanishing of the $\R^n$-valued part 
of the curvature two-form $\Omega$,
which is exactly given by the torsion of a metric connection.

Applying this whole language to the Euclidean space $\R^n$, we find
that the Cartan connection constructed above is simply the
Maurer--Cartan one-form $\omega_{SE(n)}$ of the special Euclidean group, that
is, for $g\in SE(n)$, we have
\begin{equation}
\label{eq:maurer}
  \left(\omega_{SE(n)}\right)_g= (L_{g^{-1}})_*\;,
\end{equation}
which has zero curvature. Note that the Cartan geometry contains
information both about the model space $\R^n$ and about its symmetry
group $SE(n)$. 

Keeping the Riemannian geometry example in mind, we give a general
overview of Cartan
geometries which are generalisations of Klein geometries. Every
homogeneous space can be identified with a
quotient $G/H$ of a Lie group $G$ by a subgroup $H$. The
Maurer--Cartan form $\omega_G$ on $G$ can be thought of as a
connection on the symmetry group $G$ of the homogeneous space $G/H$ with
values in the Lie algebra $\mathfrak{g}$ of
$G$. It is defined exactly as in~\eqref{eq:maurer} with $G$ instead of
$SE(n)$.
In Cartan geometry,
we replace $G/H$ by a $H$-principle bundle $P$ over a manifold $M$ and
the Maurer--Cartan form $\omega_G$ by a $\fg$-valued one-form on $P$.
We adopt
the convention to denote the Lie algebra associated with a Lie group by
the corresponding Gothic letter.
\begin{definition}
Given a smooth manifold $M$, a Lie group $G$ and a subgroup
$H\subset G$, a Cartan geometry $(P,\omega)$ on $M$ modelled on
$(\fg,\fh)$ consists of the following data.
\begin{enumerate}
\item A right principle $H$-bundle $\pi \colon P \to M$.
\item A $\fg$-valued one-form $\omega$ on $P$, called a Cartan
  connection, which satisfies
	\begin{enumerate}
    \item $\omega_p \colon T_p P \to \mathfrak{ g}$ is an isomorphism
      for all $p\in P$,
    \item $R^*_{h}\omega = \operatorname{Ad}_{h^{-1}}\omega$ for all $h\in H$, and
    \item $\omega(X^*) = X$ for all $X^* \in \Gamma(TP)$ and
      $X \in \fh$ which are related by,
      for all $f\in C^\infty(P)$ and all $p\in P$,
      \begin{equation*}
        \left(X^*f\right)(p) =
        \left.\frac{\db}{\db t}\right|_{t=0}f\left(p\exp(tX)\right).
      \end{equation*}
    \end{enumerate}
\end{enumerate}
The homogeneous space $G/H$ is called the model space for the
corresponding Cartan geometry.
\end{definition}

In the case of a Riemannian manifold $M$ of dimension $n$, a
Cartan geometry is
modelled over $(\mathfrak{se}(n),\mathfrak{so}(n))$,
the model space is given
by the corresponding quotient $\R^n \simeq SE(n)/SO(n)$, and
$P$ is the orthonormal frame bundle $\of(M)$ viewed as a $SO(n)$-principle
bundle. The connection $\omega$ constructed above then indeed
satisfies the properties of a Cartan
connection, cf. Sharpe~\cite{sharpich}. Note that the Levi--Civita
connection is
just a particular instance of a Cartan connection which is
characterised by the vanishing of the torsion part of the curvature
two-form.

The usefulness of Cartan geometries for our work arises from the property that
they possess a good notion of development of curves.
Let us fix $q\in M$.
A curve $\gamma_{G/H}\colon [0,1]\to G/H$ on the model space $G/H$ is
developed via a Cartan connection
$\omega$ to a curve $\gamma_M\colon [0,1]\to M$
with $\gamma_M(0)=q$
on the manifold $M$ as follows.
\begin{enumerate}
  \item The curve $\gamma_{G/H}$ is lifted to a curve $\gamma_G\colon
    [0,1] \to G$.
  \item Fixing some $p\in P$ such that $\pi(p)=q$, we define the
    development $\gamma_P\colon [0,1]\to P$ of
    the lift $\gamma_G$ by requiring
    \begin{equation}
      \label{eq:develop}
      \gamma_P^* \omega = \gamma_G^* \omega_G
    \end{equation}
    subject to $\gamma_P(0) = p$.
\item The development $\gamma_M$ of $\gamma_{G/H}$ is given by
  $\gamma_M = \pi(\gamma_P)$.
\end{enumerate}
The relation~\eqref{eq:develop} defines the development $\gamma_P$
of the lift $\gamma_G$
uniquely once an initial point for $\gamma_P$ is
specified. Moreover, according to
Sharpe~\cite[Proposition 5.4.13]{sharpich}
we have the following property.
\begin{theorem}
\label{thm:lift}
In a Cartan geometry $(P,\omega)$ on a manifold $M$ modelled on
$(\fg,\fh)$, the development $\gamma_M$ of a curve $\gamma_{G/H}$
does neither depend on the choice of a lift $\gamma_{G}$ nor on
the choice of a lift $p=\gamma_P(0)$ of $q=\gamma_M(0)$.
\end{theorem}

The same scheme works in the opposite way, where the
relation~\eqref{eq:develop} is used to define an anti-development on
the model space $G/H$ of a curve $\gamma_M\colon [0,1]\to M$ on the
manifold.

\subsection{Cartan geometry on a sub-Riemannian manifold}
\label{sec:scartan}
As discussed in the Introduction, a local model for an
equinilpotentisable sub-Riemannian manifold $(M,\cD,g)$ is the
nilpotent Carnot 
group $N$ with the Lie algebra $\fn = \gr(T_q M)$. It inherits the
natural grading
\begin{equation*}
  \fn= \fn_{-1} \oplus \dots \oplus \fn_{-m} 
\end{equation*}
and a scalar product $g_{-1}$ on $\fn_{-1}$. Thus, $N$ is itself a
sub-Riemannian manifold. 

In order to define a Cartan geometry on the sub-Riemannian manifold
$(M,\cD,g)$, we need to consider the Lie algebra
\begin{equation*}
  \fg = \fn \oplus \fh\;,
\end{equation*}
where $\fh$ is the Lie algebra associated with the Lie group $H$ of
all automorphisms of
$\fn$ which preserve $g_{-1}$, that is,
\begin{align}
  \begin{aligned}
    \label{eq:symm}
    \fh = \{\varphi\colon\fn \to \fn
    \mbox{ such that }
    & g_{-1}(\varphi(\cdot),\cdot)+g_{-1}(\cdot,\varphi(\cdot))=0\\
    \mbox{ and }
    & \varphi([X,Y])=[\varphi(X),Y]+[X,\varphi(Y)]
    \mbox{ for all }
    X,Y \in \fn \}\;.
  \end{aligned}
\end{align}
In particular, the Lie algebra $\fh$ is isomorphic to a sub-algebra of
$\mathfrak{so}(\fn_{-1})$.
Equivalently, we could have asked $H$ to preserve the metric on
all of $\fn$. Indeed, since the metric on $\fn$ is constructed by
identifying each $\fn_{-l}$ with a subspace of $l$ tensor products
of $\fn_{-1}$, the metric on $\fn_{-l}$ is simply given by the
induced metric from the tensor product, and as $H$ preserves the
structure of the brackets, $\fn_{-l}$ corresponds to an orthogonal
sub-representation of $H$ in $\otimes_{i=1}^l \fn_{-1}$.

We emphasise that the Lie algebra $\fg$ is a graded
space with elements of $\fh$ having degree zero and elements of
$\fn_{-i}$ having degree $-i$. Similarly, we define the degree of
elements of the dual spaces $\fn_{-i}^*$ to be $i$. This endows any
tensor product of those spaces with a grading. For example, we use
later that all elements from $\fn_{-i} \otimes \fn_{-j}^*
\wedge \fn_{-k}^*$ have degree $j+k-i$. In order to make calculations
consistent, we further treat the zero element as an element which can
take any degree. We denote by $+$ in the subscript the subspace
spanned by the elements of positive degree. 

The second ingredient needed to define a Cartan geometry is a right
principle bundle $P$.
Similar with Riemannian geometry, $P$ is a bundle of graded frames
which is formed
by all Lie algebra morphisms 
\begin{equation*}
  f \colon \fn \to \gr(TM)
\end{equation*}
compatible with the metric. This means that for any orthonormal basis 
$\{e_1,\dots,e_{k_1}\}$ of $\fn_{-1}$ the elements $X_i=f(e_i)$, for
$i\in\{1,\dots,k_1\}$, should
form an orthonormal basis of $\cD$.
The bundle $P$ has $H$ as its structure group, and it
is a reduction of the orthonormal frame bundle 
$\cO(\cD)$ to the group $H$. In what follows, we use the notation
$\cO_H(\cD)$ for this bundle $P$.

The graded frame bundle possess the canonical soldering form $\omega_\fn$ 
which is defined as follows. For all $v\in T\cO_H(\cD)$ and all 
$f\in\cO_H(\cD)$, we set
\begin{equation*}
  \left(\omega_\fn\right)_f(v) = f^{-1}\gr(\db\pi(v))\;,
\end{equation*}
where $\gr\colon TM\to\gr(TM)$ is defined by~\eqref{grref}.

Even though $\gr (T_{\pi(f)} M)$
is isomorphic to $T_{\pi(f)} M$, it is not a canonical
isomorphism. Indeed, any element of an adapted frame is defined only
modulo terms of higher degree.
This means that, for any $q\in M$, we have to choose an
additional isomorphism $T_{q}M \to \gr(T_{q}M) $ which
would allow us to decompose a vector field $X\in \Gamma(TM)$ according
to the filtration
\begin{equation*}
  \gr(X) = X_{-1}\oplus ... \oplus X_{-m}\;.
\end{equation*}
Any choice of a soldering gauge provides us with a needed automorphism. 
The canonical soldering form gives us only the degree one components of
the needed isomorphism.

\begin{remark}
  In the study of filtered manifolds, one usually has to apply a
  procedure known as the Tanaka prolongation, which consists
  of adding higher derivations of $\fg$. However, the fact that we
  study the metric geometry on equinilpotentisable sub-Riemannian
  manifolds forces
  the Tanaka prolongation to be trivial,
  see~\cite[Section~2]{constant_sub}. Hence, all information we need
  is already contained in $\fg$.
\end{remark}

For a generic sub-Riemannian manifold, a torsion-free
connection does not exist. However, we describe below how to
construct linear conditions on the curvature function which
guarantees the existence of a unique Cartan connection for a given
pair $(\fn,g_{-1})$.

\begin{definition}
The curvature function $\kappa\colon P \to \hom(\wedge^2 \fn,\fg)$ of a
Cartan connection $\omega$ is defined as
  \begin{equation*}
    \kappa(p)(\cdot,\cdot)=
    \Omega_p\left(\omega^{-1}_p(\cdot),\omega^{-1}_p(\cdot)\right)\;.
  \end{equation*}
\end{definition}

We recall that the Lie algebra differential $\p\alpha$ of
$\alpha\in \hom(\wedge^k \fn,\fg)$ is defined, for any vectors
$X_0,\dots,X_k \in \fn$, by
\begin{align*}
  \p \alpha(X_0,X_1,\dots,X_k)
  &= \sum_{i=0}^k (-1)^i X_i \cdot
    \alpha\left(X_0,\dots,\hat X_i,\dots,X_k\right) \\
  &\qquad +\sum_{0\leq i < j \leq k}(-1)^{i+j}
    \alpha\left([X_i,X_j],X_0,\dots,\hat{X}_i,\dots,\hat X_j,\dots, X_k\right)\;,
\end{align*}
where hat means the omission of the corresponding vector and where
$X_i \cdot{}$ denotes the
adjoint action $\ad$ of $X_i$. 
For a basis
$\{e_1,\dots,e_{\dim\fg}\}$ of $\fg$ and the corresponding
dual basis $\{e^1,\dots,e^{\dim\fg}\}$, the Lie algebra
differential $\p$ satisfies
\begin{equation*}
  \p\left(e_i \otimes e^j\right)
  = \p e_i \wedge e^j + e_i \otimes  \p e^j\;,
\end{equation*}
where $\p e_i = -\ad e_i $ for $e_i \in \fn$.

The construction of Cartan connections on sub-Riemannian
manifolds satisfying a type of torsion-free-like condition relies on
finding suitable normal modules of $\hom(\wedge^2 \fn,\fg)_+$. This is
also the heart of Section~\ref{sec:cartan}.

\begin{definition}
  A subspace $\cN\subset \hom(\wedge^2 \fn,\fg)_+$ is called a normal module if
  \begin{enumerate}
  \item $\cN$ is a $H$-module with respect to the adjoint action of $H$
    on  $ \hom(\wedge^2 \fn,\fg)$, and
  \item we have $\hom(\wedge^2 \fn,\fg)_+ = \cN  \oplus \im
    \p(\hom(\fn,\fg)_+)$.
  \end{enumerate}
\end{definition}

To relate Cartan connections and normal modules, we
  further need the notion of an adapted Cartan connection.

\begin{definition}
A Cartan connection $\omega$ on a sub-Riemannian
manifold $(M,\cD,g)$ is called adapted if for any arbitrary section
$s\colon M \to \cO_H(\cD)$ the $\fn$-valued part of the Cartan gauge
$s^*\omega$ forms a coframe dual to an adapted frame. 
\end{definition}

We then have the following theorem,
see Morimoto~\cite[Theorem~3.10.1]{morimoto}.

\begin{theorem}
Given an equinilpotentisable sub-Riemannian manifold $(M,\cD,g)$ and a
normal module $\cN \subset \hom(\wedge^2 \fn,\fg)_+$ there exists a
unique Cartan geometry $(\cO_H(\cD),\omega)$ on $M$ modelled on
$(\fg,\fh)$ such that the Cartan connection $\omega$ is adapted and
the corresponding curvature function $\kappa$ takes values in $\cN$.
\end{theorem}

The problem that we face is how to choose such a normal module
$\cN$. One of the possibilities is according
to the following construction due to
Morimoto~\cite{morimoto}. In the Introduction, we discuss how the induced
metric $g_{-1}$ on $\fn_{-1}$ defines a metric on all of $\fn$. We
extend it to $\fg$ by assuming that $\fh$ is endowed with a
bi-invariant metric orthogonal to $\fn$, which further
  gives rise to a metric on any tensor product of $\fg$ and its
  subspaces. Thus, we can also
define the duals of linear operators acting on these products. In
particular, we can define the adjoint map $\p^*$, and by the usual
linear algebra arguments $\ker \p^*$ is the orthogonal complement to
$\im \p$. Hence, the subspace $\cN = \ker \p^*$ gives a natural choice for a
$H$-module. However, this module does not always give rise
to the sub-Riemannian diffusion associated with the
sub-Laplacian defined with respect to the Popp volume.

In Section~\ref{sec:cartan}, we construct Cartan connections yielding the
desired sub-Riemannian diffusion by choosing the normal module $\cN$
to be orthogonal to a different module $\mathcal{S}$, which depends on the
structure of the Lie algebra $\fn$.

We end this section by discussing some natural associated bundles
related to equiregular sub-Riemannian manifolds. They are used in
Section~\ref{sec:cartan} to give invariant necessary and sufficient
conditions for the existence of Cartan connections which develop the
canonical sub-Riemannian diffusion process on the nilpotent group $N$ to the
stochastic process with generator $\frac{1}{2}\Delta_\cP$ on $M$.

Given a right principle $H$-bundle $\pi\colon P \to M$ and a
representation $\rho\colon H \to V$ for some vector space $V$, we can
construct the associated bundle $P \times_H V$.
The following proposition,
cf. Sharpe~\cite[Section 1.3]{sharpich}, is needed in the analysis in
Section~\ref{sec:cartan}.

\begin{proposition}
\label{prop:sections_and_functions}
There exists a bijection between sections of $P \times_H V$ and the
space of all equivariant functions, that is, the space
\begin{equation*}
  \left\{f\colon P \to V \mbox{ such that }
  f\in C^\infty(P,V) \mbox{ and }
  f(ph)=\rho\left(h^{-1}\right)f(p) \mbox{ for all }
  p\in P, h\in H\right\}\;.
\end{equation*}
\end{proposition}

For example, let us consider the curvature function $\kappa$.
From the equivariant properties of the Cartan connection $\omega$,
it follows that, cf.~\cite[Lemma 5.3.23]{sharpich},
\begin{equation}
\label{eq:curve_trans}
\kappa(ph)(v_1,v_2) =
\Ad\left(h^{-1}\right)\left(\kappa(p)(\Ad(h)v_1,\Ad(h)v_2)\right)
\mbox{ for all }v_1,v_2 \in \fn\;,
\end{equation}
which implies that the function $\kappa\colon P \to \hom(\wedge^2 \fn,\fg)$ is
actually a section of the associated bundle $P \times_H
\hom(\wedge^2 \fn,\fg)$.
Another example is given by the tangent
bundle $TM$, which according to~\cite[Theorem 5.3.15]{sharpich} is
isomorphic to $P \times_H \fn$. Similarly, as mentioned after
the definition of $\fh$, each $\fn_{-l}$ is an orthogonal
representation of $H$, and therefore, for $l\in\{2,\dots,m\}$,
we can consider the associated
bundle $ P \times_H (\fn_{-1}\oplus...\oplus \fn_{-l})$, which is
isomorphic to $\cD^{-l}$.

\section{Stochastic development in sub-Riemannian
  geometry}
\label{sec:develop}
The goal of this section is to introduce a notion of stochastic
development on sub-Riemannian
manifolds, and to determine the generator of
the stochastic process obtained as
the stochastic
development of a canonical sub-Riemannian diffusion on the associated
nilpotent model. To motivate our definition, we start by
discussing how horizontal differentiable curves on the model space are
developed.

We consider an arbitrary Cartan geometry
$(\cO_H(\cD),\omega)$ with an adapted Cartan connection $\omega$
on an equinilpotentisable sub-Riemannian manifold
$(M,\cD,g)$.
Let $(X_1,\dots,X_n)$ be an adapted frame on $M$ and let $e_i =
\gr(X_i)$, for $i\in\{1,\dots,n\}$, be the elements of the
corresponding adapted basis of the Lie algebra $\fn$. We
denote by $\{e^1,\dots,e^n\}$ the dual basis of $\{e_1,\dots,e_n\}$.
For $h \in H$, we write
$\rho_h$ for the corresponding action of $H$ on, depending on the
context, all of $TM$ or all of $T^*M$, and $h$ for the action on
subspaces isomorphic to $\R^k$, for example, on $\cD^{-1}$ or
$\fn_{-1}$. The Maurer--Cartan form $\omega_G$ on $G$ is given by
\begin{equation*}
  \omega_G = \sum_{i=1}^{\dim G} e_i \otimes e^i\;.
\end{equation*}

Suppose we are given a model horizontal curve
$\gamma_{N}\colon [0,1]\to N$ which we wish to develop to our
sub-Riemannian manifold. Then it must satisfy
\begin{equation*}
\dot \gamma_N = \sum_{i=1}^{k_1}u_i e_i(\gamma_N)
\end{equation*}
for some smooth functions $u_i\colon [0,1]\to\R$. The assumption that
$\gamma_{N}$
is horizontal means that
$\gamma_N^* e^i=0$ for all $e^i \notin \fn^*_{-1}$. Since
$\fg=\fn \oplus \fh$, any lift $\gamma_{G}$ of $\gamma_{N}$ to the 
Lie group $G$ is uniquely defined by its projections $\gamma_N$ and
$\gamma_H$ to $N$ and $H$, respectively, and we write $\gamma_G =
(\gamma_H,\gamma_N)$. Let us choose $\gamma_H \equiv
\id$. This choice, by Theorem~\ref{thm:lift}, does not affect the
development of $\gamma_N$, but it greatly simplifies the subsequent
computations. We obtain that
\begin{equation}
  \label{eq:dev_p}
  \gamma_G^* \omega_G = \sum_{i=1}^{k_1} e_i
  \otimes \left(\gamma_N^* e^i\right)
  = \begin{pmatrix} 
    u_1\\
    \vdots\\
    u_{k_1}\\
    0\\
    \vdots\\
    0
  \end{pmatrix}\dd t\;.
\end{equation}
Note that the $\fh$-valued part of the Maurer--Cartan
  form $\omega_G$ is zero.
Following the scheme of the development of curves discussed
  in Section~\ref{sec:cartan_geom},
we now want to compute $\gamma_{\cO_H(\cD)}^* \omega$ of a Cartan connection
$\omega$ for any curve
$\gamma_{\cO_H(\cD)}\colon [0,1]\to \cO_H(\cD)$ in the adapted orthonormal
frame bundle. Performing the computation in a local trivialisation,
we can write $\gamma_{\cO_H(\cD)} = (h,\gamma_M)$.
For $i\in\{1,\dots,n\}$ fixed, let $j\in\{1,\dots,m\}$ be
  such that $e_i \in \fn_{-j}$. Then the bundle map
$f\colon \fn \to \gr(T_{\pi(f)} M)$ is given by
\begin{equation*}
  f(e_i) = X_i + \sum_{l=1}^{k_1 +\dots +k_{j-1}} b_i^l X_l\;,
\end{equation*}
where the $b_i^l$ are smooth functions on $M$. Hence, we can identify
the bundle map $f$ with a block lower-triangular matrix
all of whose diagonal blocks are identity matrices.
We denote the transpose inverse of this matrix by $F$.

As before, let $(\theta^1,\dots,\theta^n)$ be the dual frame to
$(X_1,\dots,X_n)$, and consider the $\fn$-valued vector form
\begin{equation*}
  \theta = \begin{pmatrix}
    \theta^1\\
    \vdots\\
    \theta^n
\end{pmatrix}.
\end{equation*}
The soldering gauge $\theta_{\fn}$ is then written as
\begin{equation}
\label{eq:soldering_def}
  \theta_{\fn} = F\theta\;.
\end{equation}
We start by analysing the $\fn$-part $\omega_{\fn}$ of the Cartan connection
$\omega$. Due to~\eqref{eq:dev_p} and the defining
relation~\eqref{eq:develop} of the development, we have the following
two equalities
\begin{align}
  \gamma^*_{\cO_H(\cD)} \omega_{\fn} 
  &= \begin{pmatrix}
    u_1 \\
    \vdots \\
    u_{k_1} \\
    0\\
    \vdots \\
    0
  \end{pmatrix}\dd t\; , \label{eq:dev_n}\\
  \gamma^*_{\cO_H(\cD)} \omega_{\fh} &= 0\; \label{eq:dev_h}  
\end{align}
for the pull-backs of the $\fn$-valued part $\omega_\fn$ and of the
$\fh$-valued part $\omega_\fh$, respectively, of the considered Cartan
connection $\omega$.

Let us first derive an explicit description for $\gamma_M$
from the relation~\eqref{eq:dev_n}. We define functions $a_i \in
C^\infty([0,1])$ as
$\gamma_{M}^*\theta^i = a_i \dd t$ for each $i\in \{1,\dots,n\}$. By
explicitly writing down the left hand side of~\eqref{eq:dev_n}, we see
that
\begin{equation*}
  \gamma^*_{\cO_H(\cD)} \omega_{\fn}
  = \rho_h^{-1} \gamma^*_{M} \theta_\fn 
  = \rho_h^{-1}\left(\gamma^*_{M}F\right) \gamma^*_{M}\theta 
  = \rho_h^{-1}\left(\gamma^*_{M}F\right)\begin{pmatrix}
    a_1 \\
    \vdots \\
    a_n
  \end{pmatrix}\dd t\;. 
\end{equation*}
Comparing this with~\eqref{eq:dev_n} yields
\begin{equation}
  \label{eq:1stsys}
\rho_h^{-1} \left(\gamma^*_{M}F\right)
	\begin{pmatrix}
    a_1 \\
    \vdots \\
    a_{k_1} \\
    a_{k_1+1}\\
    \vdots \\
    a_n
  \end{pmatrix}
  = \begin{pmatrix}
    u_1 \\
    \vdots \\
    u_{k_1} \\
    0\\
    \vdots \\
    0
  \end{pmatrix}\;.
\end{equation}
It follows that we can solve the system~\eqref{eq:1stsys} for the
functions $a_i$ for $i\in\{1,\dots,n\}$. Indeed, each $\fn_{-j}$ is an
orthogonal sub-representation of $H$. Thus, each matrix $\rho_h^{-1}$
is block-diagonal. The matrix $F$ is block
upper-triangular unipotent. Hence, we can solve the above system of
equations block by block starting from the lowest rows, which give us
$a_{k_{m}+1} = \dots = a_{n} = 0$. Continuing iteratively, we find
that $a_i = 0$ for all $i\in \{k_{1}+1,\dots,n\}$. Since by
construction the first $k_1\times k_1$ minor of $F$ is the identity
matrix, we finally obtain 
\begin{equation}
\label{eq:random_link}
  a =
  \begin{pmatrix}
    a_1 \\
    \vdots \\
    a_{k_1} 
  \end{pmatrix}
  = h\begin{pmatrix}
    u_1 \\
    \vdots \\
    u_{k_1} 
  \end{pmatrix} 
  = hu\;,
\end{equation}
and, due to $\gamma_{M}^*\theta^i = a_i \dd t$, the curve $\gamma_M$
satisfies
\begin{equation*}
  \dot\gamma_M = \sum_{i=1}^{k_1} a_i X_i(\gamma_M)\;.
\end{equation*}

It remains to determine $h$ from the $\fh$-part $\omega_\fh$ of the
connection $\omega$. Let $\theta_\fh$ be a Cartan gauge of
$\omega_\fh$. If we take a basis $\{A_\alpha\colon 1\leq\alpha\leq \dim H\}$ of
$\fh$ then $\theta_\fh$ can be written as 
\begin{equation}
  \label{eq:thetafh}
  \theta_{\fh} = \sum_{\alpha=1}^{\dim H} A_\alpha \left( 
    (\tilde\Gamma^\alpha)^T \theta_\fn\right), 
\end{equation}
where $\tilde \Gamma^\alpha\colon C^\infty(M) \to \R^n$
  are called Christoffel symbols. Since $\gamma_H\equiv\id$ is constant by
assumption, we deduce from~\eqref{eq:dev_h} that
\begin{equation*}
  h^{-1}\dot h + h^{-1} \left(\gamma_M^*\theta_{\fh}\right) h = 0\;,
\end{equation*}
which, using the change of variables $\tilde{h} = h^{-1}$, simplifies
to
\begin{equation*}
  \dot{\tilde{h}} = \tilde{h} \left(\gamma_M^*\theta_{\fh}\right).
\end{equation*}
Let $\Gamma^\alpha$ denote the reduced vector
\begin{equation*}
  \Gamma^\alpha =
  \begin{pmatrix}
    \Gamma_{1}^\alpha\\
    \vdots\\
    \Gamma_{k_1}^\alpha
  \end{pmatrix}\;,
\end{equation*}
and similarly we define the differential operator $X$ on $M$ with
values in $\cD$ given by
\begin{equation*}
  X =
  \begin{pmatrix}
    X_1\\
    \vdots\\
    X_{k_1}
  \end{pmatrix}.
\end{equation*}
Since $a_i = 0$ for all $i\in \{k_{1}+1,\dots,n\}$, that is,
\begin{equation*}
  \gamma_M^*\theta_{\fn} = \begin{pmatrix}
    a_1 \\
    \vdots \\
    a_{k_1} \\
    0\\
    \vdots \\
    0
  \end{pmatrix}\;,
\end{equation*}
the expression~\eqref{eq:thetafh} simplifies to
\begin{equation*}
  \gamma_M^{*}\theta_{\fh}
  = \sum_{\alpha=1}^{\dim H}
    A_\alpha\left((\tilde \Gamma^\alpha)^T(\gamma_M)
      \left(\gamma_M^*\theta_{\fn}\right)\right)
  = \sum_{\alpha=1}^{\dim H} A_\alpha\left(a^T \Gamma^\alpha(\gamma_M)\right)\dd t\;.
\end{equation*}

Let $\{Y_\alpha\colon 1\leq\alpha\leq \dim H\}$ be the family of 
left-invariant vector fields on $H$ corresponding to the basis
$\{A_\alpha\colon 1\leq\alpha\leq \dim H\}$. In particular, we have
$Y_\alpha(\tilde{h})=\tilde{h} A_\alpha$. Using~\eqref{eq:random_link}
and the fact that $H\subset SO(k_1)$, which gives $a^T=u^T \tilde{h}$,
the
proof of the proposition stated below follows.

\begin{proposition}
\label{prop:later_ref}
Let $(M,\cD,g)$ be an equinilpotentisable sub-Riemannian manifold with
a symmetry group $H$ and a model space $N = G/H$. Let
$(\cO_H(\cD),\omega)$ be a Cartan
geometry on $M$ modelled on $(\fg,\fh)$. Choose an orthonormal basis
$(X_{1},\dots,X_{k_1})$ of $\cD$ and let $e_i = \gr(X_i)$ be the
corresponding basis of $\fn_{-1}$. Then any development $\gamma_M
\colon [0,1]\to M$ of a horizontal curve $\gamma_N \colon [0,1] \to N$
is a projection of a curve $\gamma_{\cO_H(\cD)}\colon [0,1]\to\cO_H(\cD)$
which in the chosen basis satisfies the following system of differential
equations
\begin{align}
  \label{eq:dev1}
  \begin{aligned}
    \dot{\gamma}_M
    &= u^T \tilde{h}X(\gamma_M)\;,\\
    \dot{\tilde h}
    &= \sum_{\alpha =1}^{\dim H}
    \left(u^T \tilde{h} \Gamma^\alpha(\gamma_M)\right)Y_\alpha(\tilde{h})\;,
  \end{aligned}
\end{align}
where $u_i$ are defined by $\gamma^*_N e^i =u_i \dd t$, for
$i\in\{1,\dots,k_1\}$ and $\{e^1,\dots,e^{k_1}\}$ the dual basis of
$\{e_1,\dots,e_{k_1}\}$.
\end{proposition}

\begin{remark}
We recall that Theorem~\ref{thm:lift} says that the development
$\gamma_M$ only depends on a choice of initial point
$\gamma_M(0) \in M$.
In particular, once $\gamma_M(0)$ has been chosen, the solution of
the smooth system~\eqref{eq:dev1} of ordinary differential equations
always projects to the same curve on $M$ irrespective of the initial
condition for $h$.
\end{remark}

This motivates the definition of stochastic
development we provide below for which we need one last ingredient taking the
place of the horizontal curve $\gamma_N$ we develop in the
deterministic setting. Any semimartingale $(w_t)_{t\geq 0}$ on
$\R^{k_1}$ lifts uniquely to a semimartingale $(\tilde w_t)_{t\geq 0}$
on the Carnot group $N$. In particular, the lift of Brownian motion
$(b_t)_{t\geq 0}$ on $\R^{k_1}$ is the stochastic process
$(\tilde b_t)_{t\geq 0}$ on $N$ whose generator $\frac{1}{2}\Delta$ is
given by
\begin{equation*}
  \Delta=\sum_{i=1}^{k_1} V^2_i\;,
\end{equation*}
where $V_i$ is the left-invariant vector field on $N$ corresponding to
$e_i$ for $i\in\{1,\dots,k_1\}$. As the operator $\Delta$ on the
nilpotent Lie group $N$ is the sub-Laplacian with respect to the Popp
volume, see Vigneron~\cite{vigneron}, and since in this setting the
Popp volume further coincides with the right Haar measure, the left
Haar measure and the Lebesgue measure, the lift
$(\tilde b_t)_{t\geq 0}$
on $N$ can be considered as a canonical sub-Riemannian diffusion
on $N$. 

By formally replacing time derivatives with Stratonovich differentials
and the control $u$ by the differential of the driving stochastic
process, we obtain the following definition.

\begin{definition}\label{def:stochastic}
  Let $(\tilde w_t)_{t\geq 0}$ be a semimartingale on $N$ which is the
  lift of $(w_t)_{t\geq 0}$ on $\R^{k_1}$. Then the
  stochastic development of $(\tilde w_t)_{t\geq 0}$
  is the stochastic process on the manifold $M$ which arises as
  the projection to $M$ of the unique solution to the system of
  Stratonovich stochastic differential equations, written in a local
  trivialisation,
  \begin{align*}
    \p \gamma_M &= \p w^T \tilde h X(\gamma_M)\;,\\
    \p \tilde h &= \sum_{\alpha=1 }^{\dim H}\left(\p w^T
                   \tilde{h}\Gamma^\alpha(\gamma_M)\right)
                  Y_{\alpha}(\tilde{h})\;,
  \end{align*}
  subject to a choice of initial condition.
\end{definition}

Some remarks from the stochastic analysis side are needed at this point.
The definition above does not only use the notion of stochastic
differential equations, see e.g. {\O}ksendal~\cite{oksendal} and
Rogers, Williams~\cite{rw1,rw2}, but also relies on the extension of
Stratonovich differentials to manifolds, cf. Norris~\cite{norris92}. To
avoid dwelling too deeply into the theory of stochastic calculus, we
simply provide a brief overview. A semimartingale is the right kind of
stochastic
process needed when wanting to consider a stochastic differential of
such a process. While the Stratonovich differential is not the only
stochastic differential available in stochastic calculus, it is the
one which, unlike the It\^o differential, is invariant under
coordinate transformations and is therefore more suited to differential
geometry. Moreover, as discussed in~\cite{part3essay,norris92}, it is
also important to note that Stratonovich differentials are only symbolic
and need to be understood as part of an integral equation. As a result
of this and due to the rotational invariance of Brownian motion,
even though Definition~\ref{def:stochastic} is stated in a
local trivialisation, the stochastic development of the
canonical sub-Riemannian diffusion $(\tilde b_t)_{t\geq 0}$
does not depend on this choice.

While we are mainly interested in the geometric and algebraic picture,
we need the relation from stochastic analysis that, for sufficiently nice
vector fields $Z_1,\dots,Z_k$ on $\R^N$ and Brownian motion $(b_t)_{t\geq 0}$
on $\R^k$, the unique solution $(z_t)_{t\geq 0}$ in $\R^N$
to the Stratonovich stochastic differential equation
\begin{equation*}
  \p z_t = \sum_{i=1}^k Z_i(z_t) \st b_t^i
\end{equation*}
is the stochastic process whose generator is the sum of
squares operator $\frac{1}{2}\sum_{i=1}^k Z^2_i$.

For completeness, we remark that in deriving the
formula~\eqref{eq:main_formula} given in Theorem~\ref{thm:main_formula}, we
extensively use the fact that the symmetry group $H$ is a subgroup of
the orthogonal group. Without this feature the resulting stochastic
process would not project well to the base manifold. For example, in
the Lorentzian setting the stochastic process is indeed studied as a
stochastic process on the pseudo-orthonormal frame
bundle, see Franchi and Le Jan~\cite{lorentz}. 

We are now ready to prove Theorem~\ref{thm:main_formula} stated in the
Introduction.

\begin{proof}[Proof of Theorem~\ref{thm:main_formula}]
We rewrite the system of Stratonovich stochastic differential
equations from Definition~\ref{def:stochastic} for the development of
the canonical sub-Riemannian diffusion $(\tilde b_t)_{t\geq 0}$ as
\begin{equation*}
  \p \gamma_{\cO_H(\cD)} = \p b^T Z\left(\gamma_{\cO_H(\cD)}\right)\;,
\end{equation*}
where in a local trivialisation we write
$\gamma_{\cO_H(\cD)}=(h,\gamma_M)$ and, with $X_1,\dots,X_{k_1}$ and the
$Y_\alpha$ for $1\leq\alpha\leq \dim H$ understood as vector fields on
$\cO_H(\cD)$,
\begin{equation*}
  Z = \tilde{h}X + \sum_{\alpha=1}^{\dim H}
  \left(\tilde{h}\Gamma^\alpha\right) Y_\alpha\;.
\end{equation*}
The generator of the stochastic process on $\cO_H(\cD)$ is then given by a
sum of squares operator, which in our notation can be compactly
written as 
\begin{equation*}
  \frac{1}{2}\Delta_{\cO_H(\cD)}=\frac{1}{2}Z^T Z\;.
\end{equation*}
Thus, the generator $\frac{1}{2}\Delta$ of the stochastic development
of the canonical sub-Riemannian diffusion is
given, for a function $f\in C^\infty(M)$, by
\begin{align*}
  \Delta(f)
   = \Delta_{\cO_H(\cD)}\left(\pi^* f\right)
   = Z^T Z\left(\pi^* f\right)
  & = X^T \tilde{h^T}\tilde{h}X(f) +
    \sum_{\alpha = 1}^{\dim H}(\Gamma^\alpha)^T \tilde{h}^T
    Y_\alpha(\tilde{h}) X(f)\\
  &= X^T X(f) +  \sum_{\alpha = 1}^{\dim H}
    (\Gamma^{\alpha})^T A_\alpha X(f)\;,
\end{align*}
where we used $Y_\alpha(\tilde{h}) = \tilde{h}A_\alpha$ and the
orthogonality of $\tilde{h}$.
\end{proof}

As a straightforward consequence of Theorem~\ref{thm:main_formula}, we
recover the Riemannian case. For a Riemannian manifold of dimension
$n$, we have $N=\R^n$ and $H=SO(n)$. In particular, the
  basis elements of $\mathfrak{so(n)}$ are the skew-symmetric matrices
$A_{ij}= E_{ij}-E_{ji}$, for $i,j\in\{1,\dots,n\}$, where $E_{ij}$ is
the matrix whose only non-vanishing element is the $(i,j)^{\rm th}$
entry, which is equal to one.
We denote the corresponding vectors of Christoffel symbols by
\begin{equation*}
  \Gamma^i_j = \begin{pmatrix}
    \Gamma^i_{1j}\\
    \vdots\\
    \Gamma^i_{nj}
  \end{pmatrix}\;.
\end{equation*}
We compute
\begin{align*}
  \Delta
  &= \sum_{i=1}^n X_i^2 +
    \sum_{1\leq j<k\leq n}
    \left( \Gamma^j_k\right)^T\left(E_{jk}-E_{kj}\right) X =\\
  &= \sum_{i=1}^n X_i^2 +
    \sum_{1\leq j<k\leq n} \left(\Gamma^j_{jk}X_k -\Gamma^j_{kk}X_j\right)
    = \sum_{i=1}^n X_i^2 -
    \sum_{1\leq j<k\leq n} \left(\Gamma^k_{jj}X_k +\Gamma^j_{kk}X_j\right)  = \\
  &=\sum_{i=1}^n X_i^2 -
    \sum_{1\leq j<k\leq n} \Gamma^k_{jj}X_k - \sum_{1\leq k<j\leq n} \Gamma^k_{jj}X_k = 
    \sum_{i=1}^n X_i^2 -\sum_{j\neq k} \Gamma^k_{jj}X_k = \\
  &= \sum_{i=1}^n X_i^2 -\sum_{j, k=1}^n \Gamma^k_{jj}X_k\;,
\end{align*}
where we used $\Gamma^i_{jk}=-\Gamma^k_{ji}$ in the second row, which
is a consequence of the skew-symmetry of the matrices $A_{ij}$.
On the other hand, we have
\begin{equation*}
  \divv(X_i) = \sum_{j=1}^n g(\nabla_{X_j} X_i,X_j)
  = \sum_{j=1}^n\Gamma_{ji}^j\;,
\end{equation*}
which, by skew-symmetry, yields
\begin{equation*}
  \Delta =
  \sum_{i=1}^n X_i^2 + \sum_{i=1}^n \divv(X_i)X_i
  = \sum_{i=1}^n \left( X_i^2 - \sum_{j=1}^n \Gamma^i_{jj}X_i\right)\;,
\end{equation*}
agreeing with the expression obtained above with the Cartan geometry approach.
This formula holds for any metric connection. If we now use the
Levi--Civita connection, whose Christoffel symbols can be computed in
terms of the structure constants of an orthonormal frame as
\begin{equation*}
  \Gamma^i_{jk}= \frac{1}{2}\left(c^k_{ij}-c^i_{jk}+c^j_{ki} \right)\;,
\end{equation*}
we exactly recover~\eqref{eq:sr_laplacian}. 

\section{Cartan connections for
sub-Riemannian diffusions}
\label{sec:cartan}
We prove Theorem~\ref{thm:develop} by explicitly constructing a
Cartan connection for which the stochastic development of the
canonical sub-Riemannian diffusion on the nilpotent model has
generator $\frac{1}{2}\Delta_\cP$ for $\Delta_\cP$ the
sub-Laplacian defined with respect to the Popp volume $\cP$. We
further illustrate the construction for manifolds modelled by free
nilpotent structures with two generators.

For a given Cartan geometry $(\cO_H(\cD),\omega)$ with an adapted Cartan
connection $\omega$ on an equinilpotentisable sub-Riemannian manifold
$(M,\cD,g)$, let $\Delta$ be twice the generator of the
stochastic development using the Cartan connection
$\omega$ of the canonical sub-Riemannian diffusion on the
nilpotent model $N$.
The first goal is to give an invariant
description of $\Delta - \Delta_\cP$.
Since the second order partial differential operators $\Delta$ and
$\Delta_\cP$ have the same second order term, the difference $\Delta -
\Delta_\cP$ can be understood as a vector field, which, as we see
below, is in fact
a horizontal vector field on $(M,\cD,g)$. We start by providing an
invariant description of this object, followed by giving an expression
in local coordinates.

Consider the curvature function
$\kappa\colon \cO_H(\cD) \to \hom(\wedge^2\fn,\fg)$ of the
Cartan connection $\omega$ and let $\kappa_\fn$ be its
$\hom(\wedge^2 \fn,\fn)$-valued part. As discussed in
Subsection~\ref{sec:scartan}, the curvature function $\kappa$ is
equivariant, and thus so is $\kappa_\fn$ with the same
law~\eqref{eq:curve_trans} of transformation. Note that due to the
semi-direct product structure of $G$, the adjoint action $\Ad$ of $H$
on $\fn$ coincides with the usual action of $H$. This implies that the
adjoint action of $H$ on $\hom(\wedge^2 \fn,\fn)$ coincides with the
standard action of $H$ on the isomorphic space
$\fn \otimes \fn^* \wedge \fn^*$. Hence, $\kappa_\fn$ is a section of
the associated bundle
$\cO_H(\cD) \times_H (\fn \otimes \fn^* \wedge \fn^*)$.

Let $R\colon \fn \otimes \fn^* \wedge \fn^* \to \fn_{-1} $ be the
map that is defined as follows as a composition of maps
\begin{equation}
\label{eq:def_r}
R\colon \fn \otimes \fn^* \wedge \fn^* \xrightarrow{\tr} \fn^*
\xrightarrow{\flat} \fn \xrightarrow{\pi_{-1}} \fn_{-1}\;.
\end{equation}
Here, $\tr$ is the contraction map, $\flat$ is the operation of
lowering of indices, and $\pi_{-1}$ denotes the orthogonal projection
to $\fn_{-1}$.
By Proposition~\ref{prop:sections_and_functions}, we can associate
with $\kappa_\fn$ an equivariant function $K_\fn$ with values in $\fn
\otimes \fn^* \wedge \fn^*$.
It follows that $R \circ K_\fn$ is a function with values in
$\fn_{-1}$ which is also equivariant since $H$ preserves the metric
on $\fn$ and $\fn_{-1}$ is an orthogonal sub-representation.
In particular, it is possible to associate with $R \circ K_\fn$ a
section of $\cO_H(\cD) \times_H \fn_{-1} \simeq \cD$, which we denote by
$R \circ \kappa_\fn$.

The object $R \circ \kappa_\fn$ has a simple description in a local
trivialisation. Let $\{e_1,\dots,e_n\}$ be an adapted orthonormal basis
of $\fn$, let $\{e^1,\dots,e^n\}$ be the dual basis of $\fn^*$ and let
$(X_1,\dots,X_{k_1})$ be the corresponding orthonormal frame of
$\cD$. In this trivialisation, we can write
\begin{equation*}
  \kappa_\fn = \sum_{j,k,l=1}^{n}\Omega^l_{jk} e_l \otimes e^j \wedge e^k\;.
\end{equation*}
Taking the trace, lowering the indices and using the isomorphism
between $\cD$ and $\fn_{-1}$, we obtain
\begin{equation}\label{eq:Rkappa}
  R \circ \kappa_\fn =
  \sum_{i=1}^{k_1} \sum_{j=1}^n\Omega^j_{ji} X_i\;.
\end{equation}
\begin{proposition}\label{prop:diff}
  Suppose $(M,\cD,g)$ is an equinilpotentisable sub-Riemannian
  manifold with a symmetry group $H$ and a model space $N = G/H$
  that admits a Cartan geometry $(\cO_H(\cD),\omega)$ modelled on
  $(\fg,\fh)$. Then, we have
  \begin{equation*}
    \Delta - \Delta_\cP = R \circ \kappa_\fn\;.
  \end{equation*}
\end{proposition}
\begin{proof}
  Throughout, we use summation convention over repeated indices
  for $i\in \{1,\dots,k_1\}$ and all
  other lowercase Latin indices ranging from $1$ to $n$, and
  for $\alpha \in \{n+1,\dots,n+\dim H\}$.
  The computations are performed in an arbitrary trivialisation.
  As before, let $(X_1,\dots,X_n)$ be the adapted orthonormal frame
  of $TM$ corresponding to $\{e_1,\dots,e_n\}$, let
  $\{e_{n+1},\dots,e_{n+\dim H}\}$ be a basis of $\fh$ and let
  $(\theta^1,\dots,\theta^n)$ be the coframe dual to
  $(X_1,\dots,X_n)$.

  From the formula for the curvature two-form $\Omega$, we see that
  \begin{equation}\label{eq:Omega}
    \Omega^j_{ji} =
    \left(\db \theta_\fn\right)_{ji}^j +
    \frac{1}{2}[\theta,\theta]_{ji}^j\;.
  \end{equation}
  Let us first consider the second term. We observe that
  $[\theta_\fh,\theta_\fh]$ takes values in $\fh$ and that
  $[\theta_\fn,\theta_\fn]$ consists of elements that have degree
  zero. Since the terms $\Omega^j_{ji}( e_j \otimes \theta^j_\fn
  \wedge \theta^i_\fn)$ have degree one, it follows that only the mixed
  commutators
  between $\theta_\fh$ and $\theta_\fn$ contribute to the
  expressions.
  By using~\eqref{eq:commute} and \eqref{eq:formula_com} as well as
  \begin{equation*}
    \theta_\fh=
    e_\alpha\otimes \theta^\alpha_{\fh} =
    e_\alpha\otimes \Gamma^\alpha_j \theta^j_\fn\;,
  \end{equation*}
  we obtain, for $A_\alpha$ the matrix of the action
  of $e_\alpha$ on $\fn$, that
  \begin{align}\label{eq:input1}
    \begin{aligned}
      \frac{1}{2}[\theta,\theta]^j_{ji}
      &=
      \frac{1}{2}[\theta_\fh,\theta_\fn]^j_{ji}+
      \frac{1}{2}[\theta_\fn,\theta_\fh]^j_{ji}
      =
      [\theta_\fh,\theta_\fn]^j_{ji} = \\ 
      &= \left[e_\alpha \otimes\Gamma^\alpha_k \theta^k_\fn ,e_l \otimes
        \theta^l_\fn\right]^j_{ji}
      = \left(\left[e_\alpha \Gamma^\alpha_k,e_l\right]
        \otimes \theta^k_\fn \wedge \theta^l_\fn\right)^j_{ji}=
      \Gamma^\alpha_j [e_\alpha,e_i]^j = \Gamma^\alpha_j(A_\alpha)^j_i\;.
    \end{aligned}
  \end{align}
  To deal with the first term, we recall that the
  soldering gauge $\theta_\fn$ is given by~\eqref{eq:soldering_def}.
  The important thing to remember about $F$ is that it is an
  upper-block triangular matrix whose diagonal blocks are identity matrices.
  Let $f^j_k$ and $\tilde{f}^j_k$ be
  the components of the matrices $F$ and $F^{-1}$, respectively.
  Then we can rewrite~\eqref{eq:soldering_def} as
  \begin{equation*}
    \theta^j_\fn = f^j_k \theta^k\;.
  \end{equation*}
  Thus, we have
  \begin{align*}
    \db \theta^j_\fn
    &= \db f^j_k \wedge \theta^k +  f^j_k \dd \theta^k
     = X_l\left(f^j_k\right)\theta^l\wedge \theta^k - \frac{1}{2}f^j_k
      c^k_{ls}\theta^l \wedge\theta^s = \\
    &= \left(\tilde f^l_p \tilde f^k_q X_l\left(f^j_k\right) -
      \frac{1}{2}f^j_k \tilde
      f^l_p \tilde f^s_q c^k_{ls} \right)\theta^p_\fn \wedge\theta^q_\fn
  \end{align*}
  and we deduce that
  \begin{equation*}
    \left(\db \theta_\fn\right)_{ji}^j
    = \tilde f^l_j \tilde f^k_i X_l\left(f^j_k\right) -
    \tilde f^l_i \tilde f^k_j X_l\left(f^j_k\right) - \frac{1}{2}f^j_k \tilde
    f^l_j \tilde f^s_i c^k_{ls} + \frac{1}{2}f^j_k \tilde f^l_i \tilde
    f^s_j c^k_{ls}\;.
  \end{equation*}
  The expression on the right hand side simplifies significantly
  due to $F$ and $F^{-1}$ being block upper-triangular
  matrices with identity matrices as diagonal blocks.
  For the first term, we use that the elements of the $i^{\rm th}$
  column of $F$ and $F^{-1}$ satisfy $f^k_i=\tilde{f}^k_i=\delta^k_i$
  to deduce that
  \begin{equation*}
    \tilde f^l_j \tilde f^k_i X_l\left(f^j_k\right)
    =\tilde f^l_j \delta^k_i X_l\left(f^j_k\right)
    =\tilde f^l_j X_l\left(f^j_i\right)
    =\tilde f^l_j X_l\left(\delta^j_i\right)=0\;,
  \end{equation*}
  whereas for the second term, the property $f^j_k=\tilde{f}^j_k=0$ for $k<j$
  yields
  \begin{equation*}
    \tilde f^l_i \tilde f^k_j X_l\left(f^j_k\right)
    =\tilde f^l_i X_l\left(f^j_j\right)
    =\tilde f^l_i X_l\left(n\right)=0\;.
  \end{equation*}
  For the last two terms, we exploit
  $f^j_k \tilde f^l_j=\delta^l_k$ and the antisymmetry of the
  structure constants to obtain that
  \begin{equation*}
    \left(\db \theta_\fn\right)_{ji}^j
    = - \frac{1}{2}f^j_k \tilde
    f^l_j \tilde f^s_i c^k_{ls} + \frac{1}{2}f^j_k \tilde f^l_i \tilde
    f^s_j c^k_{ls}
    = - \frac{1}{2}\tilde f^s_i c^l_{ls}
    +\frac{1}{2}\tilde f^l_i c^s_{ls}=-\tilde f^s_i c^l_{ls}\;.
  \end{equation*}
  Using once more that $\tilde{f}^k_i=\delta^k_i$ as observed
  above, we find
  \begin{equation}\label{eq:input2}
    \left(\db \theta_\fn\right)_{ji}^j = - c^l_{li}\;.
  \end{equation}
  Inserting~\eqref{eq:input1} and \eqref{eq:input2} into
  \eqref{eq:Omega} gives
  \begin{equation*}
    \Omega^j_{ji}=\Gamma^\alpha_j(A_\alpha)^j_i- c^l_{li}\;,
  \end{equation*}
  and the claimed result follows from Theorem~\ref{thm:main_formula}
  as well as formulas~\eqref{eq:sr_laplacian} and \eqref{eq:Rkappa}.
\end{proof}

We are now ready to prove Theorem~\ref{thm:develop}.

\begin{proof}[Proof of Theorem~~\ref{thm:develop}]
  Let us first prove the necessary part for the
  existence. The soldering form $\omega_\fn$ gives an isomorphism between
  $\cD$ and
  $\fn_{-1}$. If $H\subset SO(k_1)$ has a one-dimensional representation,
  then there exists $v\in \fn_{-1}$ such that $Hv = \pm v$, and
  consequently $Av = 0$ for all $A\in \fh$. Take $X\in \Gamma(\cD)$ such
  that $\omega_\fn(X) = v$ and complete it to an orthonormal adapted frame.
  Then $\omega_\fn^{-1}(v) = X$, $AX = 0$ and the vector field $X$
  does not
  appear in the divergence part of the
  formula~\eqref{eq:main_formula}.

  It remains to prove that this condition is also sufficient.
  We fix an orthonormal adapted frame $\{e_1,\dots,e_{n+\dim H}\}$ in
  $\fg$ where the first $n$ elements form a basis of
  $\fn$ and the last $\dim H$ elements a basis of $\fh$.
  As before, we assume that $\{e_1,\dots,e_{k_1}\}$ forms an
  orthonormal basis of the space $\fn_{-1}$. Let
  \begin{equation*}
    \ker \fh=\{v\in \fn_{-1}\mbox{ such that }
    Av=0\mbox{ for all }A\in\fh
    \}.
  \end{equation*}
  We can suppose that
  $\ker \fh = \operatorname{span}\{e_{{k_0} +
  1},\dots,e_{k_1}\}$, for some $k_0\in\{0,\dots,k_1\}$,
  and that all vector fields corresponding to $\ker \fh$ are
  divergence-free.
  In the following, we abuse notation and we use $g$ to refer to the
  extended metric on $\fg$ with $g_{ij}$ and $g^{ij}$, for
  $i,j\in\{1,\dots, n+\dim H\}$, denoting the components of $g$ and
  of the corresponding metric on the dual space $\fg^*$, respectively.

  Since the spaces $\fn_{-l}$ are pairwise orthogonal and using the
  orthonormality of $\{e_1,\dots,e_{k_1}\}$, we obtain, for
  $i\in\{1,\dots,k_1\}$ and $k,l,s \in \{1,\dots,n\}$,
  \begin{equation}\label{eq:metric}
    g\left(\sum_{j=1}^n e_j \otimes e^j \wedge e^i, e_k \otimes e^l
      \wedge e^s \right)= \frac12\sum_{j=1}^n
    g_{jk}\left(g^{jl}g^{is}-g^{js}g^{il}\right) =
    \frac12(\delta^l_k\delta^{is}-\delta^s_k\delta^{il})\;.
  \end{equation}
  It follows that, for $k\not= i$,
  $$
  g\left(\sum_{j=1}^n e_j \otimes e^j \wedge e^i, e_k \otimes e^k
    \wedge e^i \right) = - g\left(\sum_{j=1}^n e_j \otimes e^j
    \wedge e^i, e_k \otimes e^i \wedge e^k \right) = \frac12\;.
  $$
  and that the scalar product of $\sum_{j=1}^n e_j \otimes e^j \wedge e^i$
  with any other element of $\hom(\wedge^2
  \fn,\fg)_+$ is zero. Therefore, we obtain
  \begin{equation*}
    g\left(\sum_{j=1}^n e_j \otimes e^j \wedge e^i, \kappa \right) = \frac12
    \sum_{j=1}^n \Omega^j_{ji}\;.
  \end{equation*}
  From $\ker \fh = \operatorname{span}\{e_{{k_0} +
  1},\dots,e_{k_1}\}$ and the divergence-free assumption, we obtain
  $\sum_{j=1}^n\Omega^j_{ji}=0$ for all $i\in \{k_0+1,\dots,k_1\}$.
  Thus, if we consider the adapted Cartan connection associated with a
  normal module $\cN$ which lies in the orthogonal complement
  of the module
  \begin{equation*}
    \cS=\spann\left\{\sum_{j=1}^n e_j \otimes e^j \wedge e^i \text{ for }
      1\leq i \leq k_0\right\}
  \end{equation*}
  then, due to the curvature function $\kappa$ taking values in $\cN$,
  we have
  \begin{equation}
  	\label{eq:vanishing_curv}
    \sum_{j=1}^n \Omega^j_{ji}=0\;.
  \end{equation}
  The desired result would then follow from
  Proposition~\ref{prop:diff}.

  To complete the proof, we need to show that we can find such a
  normal module $\cN$. We recall that a module is called normal
  if it is complement to $\p(\hom(\fn,\fg)_+)$ in
  $\hom(\wedge^2 \fn,\fg)_+$. We are looking for
  a normal module $\cN\subset \cS^\perp$ such that $\cN\oplus \im \p_+
  = \hom(\wedge^2 \fn,\fg)_+$ or, equivalently, we need $\cN$ to satisfy
  \begin{equation*}
  \cS\subset\cN^\perp\quad\mbox{and}\quad
  \cN^\perp \oplus (\im \p_+)^\perp = \hom(\wedge^2 \fn,\fg)_+\;.
  \end{equation*}
  If $\cS\cap(\im \p_+)^\perp=\{0\}$, then we can take 
  \begin{equation*}
  \cN^\perp = \cS \oplus \left(\cS \oplus (\im \p_+)^\perp\right)^\perp.
  \end{equation*}
  Indeed, under this assumption, by construction, we have
  $\cN^\perp \cap (\im \p_+)^\perp = \{0\}$ and a dimension count
  shows that
  $$
  \dim \cN^\perp = \dim \hom(\wedge^2 \fn,\fg)_+ - \dim  (\im \p_+)^\perp.
  $$
  Due to $\cS$ and $(\im \p_+)^\perp$ being $\fh$-submodules, $\cN$ is
  a $\fh$-submodule as well.
  
  Let us now prove that $\cS\cap(\im \p_+)^\perp=\{0\}$. Since we can relabel
  the basis vectors $e_i$ for $i\in\{1,\dots,k_0\}$, without loss of
  generality,
  it suffices to prove that there exists some $v\in\im\p_+$ such that
  \begin{equation*}
  g \left( \sum_{j=1}^n e_j \otimes e^j \wedge e^1, v \right) \neq 0\;.
  \end{equation*} 
  As $\ker \fh = \operatorname{span}\{e_{{k_0} + 1},\dots,e_{k_1}\}$,
  there exists some $h_1 \in \fh$ where
  $h_1 \cdot e_1 = \sum_{j=1}^n \alpha^j e_j$ has at
  least one non-zero coefficient $\alpha^i$ for
  $i\in\{2,\dots,k_0\}$. Recall that due to the
  filtered algebra structure, we have $\p e^1 = 0$, and by using
  formula~\eqref{eq:metric}, we obtain 
  \begin{equation*}
    g \left( \sum_{j=1}^n e_j \otimes e^j \wedge e^1, \p(h_1\otimes
    e^i)\right) = g \left( \sum_{j=1}^n e_j \otimes e^j \wedge e^1,
    \sum_{l=1}^n \alpha_l e_l\otimes e^1 \wedge e^i\right) = -
  \alpha^i \neq 0\;,
  \end{equation*} 
  as required.
\end{proof}

A natural question that comes to mind is whether the Cartan connection
constructed by Morimoto allows us to obtain the
sub-Laplacian $\Delta_\cP$ defined with respect to the Popp
volume via the discussed stochastic development procedure. As we have seen
in the above proof, for the stochastic development of the canonical
sub-Riemannian diffusion on the model space to have generator
$\frac{1}{2}\Delta_\cP$, the curvature function needs to
satisfy
\begin{equation}\label{eq:cond4kappa}
  g\left(\sum_{j=1}^n e_j \otimes e^j \wedge e^i, \kappa \right) =0\;.
\end{equation}
This means that if we impose Morimoto's normalisation and 
the difference $\Delta - \Delta_\cP$ vanishes,
then $\sum_{j=1}^n e_j \otimes e^j \wedge e^i \in \im \p$.
Since $\p e^i = 0$ for all $i\in\{1,\dots k_1 \}$,
a necessary condition for the latter is
\begin{equation}\label{eq:cond4morimoto}
  \p\left(\sum_{j=1}^n e_j \otimes e^j\right)\wedge e^i=0
  \quad\mbox{for all }i\in\{1,\dots,k_1\}\;,
\end{equation}
which is easy to check in practice.
In Section~\ref{sec:ttfive}, we use condition~\eqref{eq:cond4morimoto}
to establish that if Morimoto's normalisation condition is chosen
for a free structure with two generators which is not a 3D
contact structure
then we have $\Delta - \Delta_\cP \neq 0$.

Examples where a stochastic development of the canonical
sub-Riemannian diffusion on the model space never gives rise to the
stochastic process with generator $\frac{1}{2}\Delta_\cP$ are easy
to find.
Let us consider a Goursat manifold, that is, a
sub-Riemannian manifold $(M,\cD,g)$ with growth vector
$(2,3,\dots,n-1,n)$ for some $n\in\N$ with $n\geq 4$.
The associated Levy form $\cL$ is a map
\begin{equation*}
  \cL \colon \cD^{-2} \times \cD^{-2} \to \cD^{-3}/\cD^{-2}
\end{equation*}
defined pointwise as follows. For vectors $v,w\in\cD_q^{-2}$ and vector fields
  $X_v,X_w \in \cD^{-2}$ extending $v$ and $w$, respectively, we set
\begin{equation*}
  \cL_q(v,w) = [X_v,X_w](q) \mod \cD_q^{-2}\;.
\end{equation*}

The forms $\cL_q$ are skew-symmetric bilinear forms
on odd-dimensional spaces and hence, they must have non-trivial
kernels $L_q$ which form a line field $L$. 
The non-integrability condition implies that $L\subset
\cD$, cf.~\cite{ivansasha}. The presence of the characteristic line
field $L$ breaks the
$SO(2)$ symmetry and leaves us with $H=\{\id\}$. This implies that
$\fh=\{0\}$ and therefore, the $\fh$-part $\omega_\fh$ of the Cartan
connection is trivial. Thus, no matter what
Cartan connection we choose, the generator of the
  developed stochastic process always ends up having vanishing first
  order term and we are left with the sum of squares term.

There is a simple way to generate many explicit examples
of this kind, because Goursat distributions often arise as Cartan
distributions in jet bundles. For instance, let us consider a
two-dimensional manifold $M$ with a global frame $(X_1,X_2)$ of
vector fields. We define a contact distribution $\cD_1$ on the direct
product $M\times S^1$ which is the span of the two vector
fields
\begin{equation*}
  Y_1 = \frac{\p}{\p \theta_1}\;, \qquad Y_2 = \cos \theta_1 X_1 +
  \sin \theta_1 X_2\;, 
\end{equation*}
where $\theta_1$ is a coordinate on $S^1$. We then apply the same
procedure a second time but this time to the pair $(Y_1,Y_2)$ of
vector fields to
obtain an Engel distribution $\cD_2$ on $M\times S^1\times S^1$
spanned by the vector fields
\begin{equation*}
  Z_1 = \frac{\p}{\p \theta_2}\;, \qquad Z_2 = \cos \theta_2 Y_1 +
  \sin \theta_2 Y_2\;,
\end{equation*}
where $\theta_2$ is a coordinate on the newly added circle $S^1$. We
can carry on with this prolongation procedure and at each iteration it
takes a Goursat manifold of step $n$ and gives us a Goursat manifold
of step $n+1$.

If we consider the upper half-plane $\R^2_+$ with coordinates
$(x,y)$ for $y>0$ and the vector fields
\begin{equation*}
  X_1 = y\frac{\p}{\p x}\;, \qquad X_2 = y\frac{\p}{\p y}\;,
\end{equation*}
then after applying the prolongation procedure twice, we find the two
vector fields $Z_1,Z_2$ that span $\cD_2$ of $\R^2_+ \times
\T^2$. Assuming that $Z_1$ and $Z_2$ are orthonormal, we obtain a
sub-Riemannian structure on $(\R^2_+ \times \T^2,\cD_2)$. Setting
\begin{equation*}
  Z_3 = [Z_1,Z_2] \quad \mbox{and} \quad Z_4 = [Z_2,Z_3]\;,
\end{equation*}
we find
\begin{equation*}
  [Z_2,Z_4] =
  -\left( \cos \theta_1\sin \theta_2 + \cos\theta_2\right)
  \left(\sin\theta_2 Z_2 + \cos \theta_2 Z_3\right) + \sin
  \theta_1\sin \theta_2 Z_4\;.
\end{equation*}
In particular, we see that
$c_{21}^1=c_{22}^2=c_{23}^3=0$ whereas
$c_{24}^4 \neq 0$ and hence, by formula~\eqref{eq:sr_laplacian}, the
sub-Laplacian $\Delta_\cP$ defined with respect to the Popp
volume is not a sum of squares operator.

Despite this example, we want to stress that there are plenty of
geometrically
interesting structures where all sub-representations of $H$ on
$\fn_{-1}$ have dimension strictly greater than one. In particular,
the existence of a Cartan connection which develops the canonical
sub-Riemannian diffusion on the model space to the stochastic process
with generator $\frac{1}{2}\Delta_\cP$ is guaranteed by
Theorem~\ref{thm:develop}, and we can follow the proof of
Theorem~\ref{thm:develop} to construct such a Cartan connection.
This is demonstrated in Section~\ref{sec:ttfive} for free
sub-Riemannian structures with two generators.

Before that, we illustrate the full Cartan machinery in
Section~\ref{sec:3d} and show that in fact any adapted Cartan
connection on a 3D contact structure satisfies
condition~\eqref{eq:cond4kappa}. As a result each stochastic development
of the canonical sub-Riemannian diffusion on the Heisenberg group has
generator $\frac{1}{2}\Delta_\cP$.

It is worth noting that the Cartan connection constructed by
Doubrov and Slov\'{a}k in~\cite{slovak_doubrov} for step two free
structures also satisfies condition~\eqref{eq:cond4kappa}.

\subsection{The three-dimensional contact case}
\label{sec:3d}
The discussions in this subsection establish the following result.
\begin{proposition}
For any Cartan connection on a 3D contact structure
which is constructed by using some normal module $\cN$,
the stochastic
development of the canonical sub-Riemannian diffusion process on the
3D Heisenberg group has generator $\frac{1}{2}\Delta_\cP$.
\end{proposition}

For 3D contact
manifolds, $\fn$ is the 3D Heisenberg Lie algebra, $\fh$ is isomorphic
to $\mathfrak{so}(2)$ and $\fg$ is a semi-direct product of the
two. The Lie algebra $\fn$ admits a basis $\{e_1,e_2,e_3\}$ which
satisfies
$$
[e_1,e_2]=e_3\;,
$$
with the other commutators in $\fn$ being zero.
The first step is to determine the commutation relations in the Lie
algebra $\fg$. Let $e_4$ be the only non-trivial element of $\fh$.
Its
action on $\fn$ is characterised by the Lie algebra derivation
condition in~\eqref{eq:symm}, that is, we need to have
\begin{equation*}
e_4([X,Y]) =[e_4(X),Y]+[X,e_4(Y)]
\end{equation*}
for all $X,Y \in \fn$.
To determine the infinitesimal action of $e_4$, we start with its
action on $\fn_{-1}$ which is simply given by the infinitesimal
rotation such that $e_4(e_1)=-e_2$ and $e_4(e_2)=e_1$. We then apply
the formula above to obtain the action of $e_4$ on $e_3$ which yields
$$
e_4(e_3) = e_4\left([e_1,e_2]\right) = [e_4(e_1),e_2]+[e_1,e_4(e_2)] = 0\;. 
$$
In particular, we can view the action of $e_4$ as the adjoint action
on the Lie algebra $\fn\subset \fg$. This gives rise to the following
structure equations on $\fg$
\begin{equation}
  \label{eq:heisenberg_relations}
  [e_1,e_2] = e_3\;, \qquad
  [e_4,e_1] = -e_2\;, \qquad
  [e_4,e_2] = e_1\;,
\end{equation}
with the remaining commutators being zero.

If we take $X_3$ to be the Reeb 
field and if $(X_1,X_2)$ is an
orthonormal frame of the contact distribution, then we have
the structure equations
\begin{align*}
  [X_1,X_2] &= X_3 + c^1_{12}X_1 + c^2_{12}X_2\;,\\
  [X_3,X_1] &= c^1_{31}X_1 + c^2_{31}X_2\;,\\
  [X_3,X_2] &= c^1_{32}X_1 + c^2_{32}X_2\;,
\end{align*}
or, by duality, for the coframe
  $(\theta^1,\theta^2,\theta^3)$ dual to $(X_1,X_2,X_3)$, we obtain
\begin{align*}
\db\theta^1 &= -c^1_{12}\theta^1\wedge \theta^2
              -c^1_{13}\theta^1\wedge \theta^3 -c^1_{23}\theta^2\wedge
             \theta^3\;,\\
\db\theta^2 &= -c^2_{12}\theta^1\wedge \theta^2 -c^2_{13}\theta^1\wedge
            \theta^3 -c^2_{23}\theta^2\wedge \theta^3\;,\\ 
\db\theta^3 &= - \theta^1\wedge \theta^2\;.
\end{align*}

For an arbitrarily chosen Cartan connection $\omega$ on $\cO_H(\cD)$,
we write
\begin{equation*}
  \omega = \sum_{i=1}^4 e_i \otimes \omega^i\;.
\end{equation*}
The first three one-forms $\omega^1,\omega^2,\omega^3$
are components of the the soldering form $\omega_\fn$,
while $\omega^4$ is $\omega_\fh$.
We can express these forms in terms of the coframe
$(\theta^1,\theta^2,\theta^3)$, some coefficients $f^1_3,f^2_3$ and
the Christoffel symbols defined in Section~\ref{sec:develop}
as follows
\begin{align*}
  \theta^i_\fn &= \theta^i + f^i_3 \theta^3\;,
                 \quad\mbox{for } i\in\{1,2\}\;, \\
  \theta^3_\fn &= \theta^3\;, \\
  \theta_\fh^4 &= \Gamma^4_1 \theta^1_\fn + \Gamma^4_2 \theta^2_\fn
                 +\Gamma^4_3 \theta^3_\fn\;.
\end{align*}
Note that the coefficients $f^1_3$ and $f^2_3$
are nothing but the non-trivial
off-diagonal components of the matrix-valued function
$F$ introduced
in Section~\ref{sec:develop}. They arise because the soldering form is
not canonically defined. The
curvature two-form $\Omega$ corresponding to $\omega$ is given by
\begin{equation*}
  \Omega = \sum_{i=1}^4 e_i \otimes \Omega^i\;,
\end{equation*}
where
\begin{align*}
  \Omega^1 &= \db\theta_\fn^1 + \theta_\fh^4 \wedge \theta_\fn^2\;,\\
  \Omega^2 &= \db\theta_\fn^2 - \theta_\fh^4 \wedge \theta_\fn^1\;,\\
  \Omega^3 &= \db\theta_\fn^3 + \theta_\fn^1 \wedge \theta_\fn^2\;,\\
  \Omega^4 &= \db\theta_\fh^4\;,
\end{align*}
follows from~\eqref{eq:curvature}.
We determine the components $\Omega^i_{jk}$ by putting the
formulae for the exterior differentials $\db\theta^1$,
$\db\theta^2$ and $\db\theta^3$ into the above expressions and by
reading off the coefficients in front of $\theta_\fn^j \wedge \theta_\fn^k$.
In particular, we find
\begin{align*}
\Omega_{12}^2 &= -c^1_{12}-f^1_3 + \Gamma^1_4\;, & \Omega_{32}^2 &= f^1_3\;, \\
\Omega_{21}^1 &= c^2_{12}+f^2_3 - \Gamma^2_4\;, & \Omega_{31}^1 &= -f^2_3\;.
\end{align*}
The condition~\eqref{eq:vanishing_curv} implies that
the Christoffel symbols of 
the Cartan connection constructed in the proof of
Theorem~\ref{thm:develop}
satisfy $\Gamma^1_4 = c^1_{12}$ and $\Gamma^4_2 = c^2_{12}$.
Therefore, the difference $\Delta - \Delta_\cP$ does indeed vanish.

An interesting property of 3D contact structures
is that we have $\Delta - \Delta_\cP=0$ for any
choice of normal module $\cN$.
We argue as follows.
As a consequence of the relations~\eqref{eq:heisenberg_relations},
the basic differentials are given by
\begin{align*}
\p e_1 &= -e_3 \otimes e^2\;,\\
\p e_2 &= e_3 \otimes e^1 \;,\\
\p e_3 &= 0\;,  \\
\p e_4 &= e_2 \otimes e^1 - e_1 \otimes e^2\;, \\
\p e^3 &= -e^1\wedge e^2\;, \\
 \p e^1 &= \p e^2= 0\;.
\end{align*}
Recall that $\deg e_1 =\deg e_2 = -1$, $\deg e_3 = -2$, $\deg e_4 = 0$
and the opposite signs for the upper index. Therefore, the differential
preserves the grading of the spaces $\hom(\wedge^k
  \fn,\fg)$.

To verify condition~\eqref{eq:cond4kappa},
we only need to compute the components of degree
one. We find that
\begin{align*}
\p \left(e_4 \otimes e^1\right) &= -e_1 \otimes e^2 \wedge e^1\;, \\
\p \left(e_4 \otimes e^2\right) &= e_2 \otimes e^1 \wedge e^2\;, \\
\p \left(e_1 \otimes e^3\right) &= -e_3 \otimes e^2 \wedge e^3\;, \\
\p \left(e_2 \otimes e^3\right) &= e_3 \otimes e^1 \wedge e^3\;, 
\end{align*}
which shows that on degree one forms the Lie algebra
differential $\p$ is a bijection. Hence, any normal module has no
degree one components as it is transversal to $\im\p$.
Since the curvature function $\kappa$ takes values in the normal
module, it follows that also $\kappa$ has no degree one components, 
and
condition~\eqref{eq:cond4kappa} is satisfied automatically.

In particular, the Cartan connection built using Morimoto's
normalisation agrees with the Cartan connection constructed in the
proof of Theorem~\ref{thm:develop}.

\subsection{Free sub-Riemannian structures with
  two generators}
\label{sec:ttfive}
The Cartan connection that we construct in the proof of
Theorem~\ref{thm:develop} looks particularly simple in the case of
free structures with two generators.
For these structures, the nilpotent Lie algebra $\fn$ is a free
nilpotent Lie algebra with generators $e_1,e_2$, and the Lie algebra
$\fh$ is generated by a single element $e_0$.

We recall that in the proof of Theorem~\ref{thm:develop} we construct
the normal module $\cN$ by considering
$\p (\hom(\fn,\fg)_+)$ and by replacing certain $k_1$ elements from
$\p \hom(\fn_{-1},\fh)$ with $\sum_{j=1}^n e_j \otimes e^j \wedge e^i$
for $i\in \{1,\dots,k_1\}$. If we have only two generators, that is
$k_1=2$, then
\begin{equation*}
  \cS=\spann
  \left\{\sum_{j=1}^n e_j \otimes e^j \wedge e^1,
    \sum_{j=1}^n e_j \otimes e^j\wedge e^2\right\}
\end{equation*}
and
\begin{equation*}
  \dim \left(\p \hom(\fn_{-1},\fh)\right)
  = \dim\cS = 2\;.
\end{equation*}
Following the construction in the proof of Theorem~\ref{thm:develop},
we take the normal module $\cN$ to be orthogonal to
\begin{equation*}
  \spann\left\{\sum_{j=1}^n e_j \otimes e^j \wedge e^1,
    \sum_{j=1}^n e_j \otimes e^j\wedge e^2,
    \p \hom\left(\fn_{-i-1},\fn_{-i}\right)
    \mbox{ for } 1\leq i \leq m-1\right\}.
\end{equation*}

As illustrated below this module need not be the only
possible choice for the normal module $\cN$.
For instance, let us consider structures with growth vector $(2,3,5)$, that
is, the space $\fn$ is spanned by $e_1,e_2,e_3,e_4,e_5$ and we have
the following structure equations on $\fg=\fn\oplus \fh$
\begin{align*}
  [e_1,e_2] & = e_3\;, & [e_1,e_3] &= e_4\;, & [e_2,e_3]&=e_5\;,\\
  [e_0,e_1] & = -e_2\;, & [e_0,e_4] & = -e_5\;, \\ 
  [e_0,e_2] & = e_1\;, & [e_0,e_5] & = e_4\;.
\end{align*}
The differentials of degree one components are given, for $e_i
\otimes e^j \wedge e^k$ abbreviated to $e_i^{jk}$, by
\begin{align*}
  \p \left(e_0 \otimes e^1\right) &= -e_1^{21}+e_5^{41}-e_4^{51}\;,\\
  \p \left(e_0 \otimes e^2\right) &= e_2^{12}+e_5^{42}-e_4^{52}\;,\\
  \p \left(e_1 \otimes e^3\right) &= -e_3^{23}-e_1^{12}\;,\\
  \p \left(e_2 \otimes e^3\right) &= e_3^{13}-e_2^{12}\;,\\
  \p \left(e_3 \otimes e^4\right) &= e_4^{14}+e_5^{24}-e_3^{13}\;,\\
  \p \left(e_3 \otimes e^5\right) &= e_4^{15}+e_5^{25}-e_3^{23}\;.
\end{align*}
Then we can take the normal module $\cN$ to
be orthogonal to
\begin{equation*}
\spann
  \left\{e_1^{12},e_2^{12},e_3^{13},e_3^{23},
    e_4^{i4}+e_5^{i5}\mbox{ for } i=1\mbox{ and }i=2\right\}
\end{equation*}
because this span contains $\cS$ as a submodule.

We close by noting that a simple calculation establishes
\begin{equation}
\label{eq:morimotofinal}
  \p\left( \sum_{j=1}^5 e_j \otimes e^j\right) \wedge e^1
  = - e_5 \otimes e^3 \wedge e^2 \wedge e^1 \neq 0\;.
\end{equation}
Thus, the condition~\eqref{eq:cond4morimoto} is not satisfied,
which implies that under Morimoto's normalisation the operator
$\Delta$ does not coincide with $\Delta_\cP$.
Moreover, we can use the observation~\eqref{eq:morimotofinal} to show
that the normalisation of Morimoto gives $\Delta - \Delta_\cP \neq 0$
for any free structure with two generators and step strictly greater
than two. Indeed, for higher step structures, for $i\in\{1,\dots,5\}$,
the expressions for the $\p e^i$ agree with the ones in the $(2,3,5)$
case, whereas the expressions for the $\p e_i$ only differ from the
ones in the $(2,3,5)$ case by higher order terms.
Hence, the differential
\begin{equation*}
  \p\left( \sum_{j=1}^n e_j \otimes e^j\right) \wedge e^1
\end{equation*}
agrees with~\eqref{eq:morimotofinal} modulo some terms involving
elements from higher steps. In particular,
it does not vanish.

\bibliographystyle{siam}
\bibliography{references}

\end{document}